\documentclass[11pt, oneside]{amsart} 

\usepackage[english]{babel} 
\usepackage{graphics,color,pgf}
\usepackage{epsfig}
\usepackage[ansinew]{inputenc}
\usepackage[all]{xy}
\usepackage{tikz-cd}
\usepackage{bbm}
\usetikzlibrary{matrix,arrows,decorations.pathmorphing}
\newdir{ >}{!/8pt/@{}*@{>}}
\usepackage{amssymb, amsmath,amsthm, mathtools, amscd}
\usepackage{mathrsfs}
\usepackage[margin=1.5in]{geometry}
\usepackage{tikz}
\usetikzlibrary{automata, arrows.meta, positioning}
\usepackage[pagebackref, pdfborder={0 0 0}
  ,urlcolor=black,hypertexnames=false]{hyperref}

\usepackage[margin=1.5in]{geometry}
  
\makeatletter
\newtheorem*{rep@theorem}{\rep@title}
\newcommand{\newreptheorem}[2]{%
\newenvironment{rep#1}[1]{%
 \def\rep@title{#2 \ref{##1}}%
 \begin{rep@theorem}}%
 {\end{rep@theorem}}}
\makeatother

\newtheorem{intro_thm}{Theorem}

\newtheorem{intro_prop}[intro_thm]{Proposition}

\newtheorem{lemma}{Lemma}[section]
 
\newtheorem{prop}[lemma]{Proposition}
\newtheorem{cor}[lemma]{Corollary}

\theoremstyle{definition}
\newtheorem{defn}[lemma]{Definition}
\newtheorem{es}[lemma]{Example}

\theoremstyle{remark}
\newtheorem{oss}[lemma]{Remark}

\newtheoremstyle{TheoremNum}
        {0.2 cm}{0.2 cm}              %%% space between body and thm
        {\itshape}                      %%% Thm body font
        {}                              %%% Indent amount (empty = no indent)
        {}                     %%% Thm head font
        {.}                             %%% Punctuation after thm head
        { }                             %%% Space after thm head
        {\thmname{\bfseries #1}\thmnote{ \bfseries #3}}%%% Thm head spec
    \theoremstyle{TheoremNum}
   
\newtheorem{rec_thm}{Theorem}
\newtheorem{rec_prop}[rec_thm]{Proposition}

\newcommand\matR{{\mathbb{R}}}

\newcommand\matN{{\mathbb{N}}}
\newcommand\matZ{{\mathbb{Z}}}

\newcommand{\id}{\mathrm{id}}

\newcommand\calL{{\mathcal L}}

\newcommand\calB{{\mathcal B}}

\newcommand\calG{{\mathcal G}}
\newcommand\calH{{\mathcal H}}

\newcommand\calR{{\mathcal R}}

\newcommand{\Hm}{\textup{H}}
\newcommand{\Hb}{\textup{H}_{\textup{b}}}
\newcommand{\Hmb}{\textup{H}_{\textup{mb}}}

\newcommand{\Hcb}{\textup{H}_{\textup{cb}}}

\newcommand{\Linf}{\text{L}^{\infty}}

\newcommand{\Lone}{\text{L}^{1}}

\newcommand{\Ima}{\text{Im}}

\newcommand{\Isom}{\text{Isom}}

\newcommand{\frakm}{\mathfrak{m}}

\newcommand{\Linfw}{\textup{L}_{\textup{w}^*}^{\infty}}

\begin{document}

\title[Measurable bounded cohomology of measured groupoids]{Measurable bounded cohomology of measured groupoids}

\author[F. Sarti]{F. Sarti}
\address{Department of Mathematics, University of Pisa, 
Largo Bruno Pontecorvo, 5, Pisa 56127, Italy}
\email{filippo.sarti@dm.unipi.it}

\author[A. Savini]{A. Savini}
\address{Department of Mathematics, University of Milano-Bicocca, Via Roberto Cozzi, 55, Milano 20126, Italy}
\email{alessio.savini@unimib.it}

\date{\today.\ \copyright{\ F. Sarti, A. Savini}.}
%The second  author 
%%(ORCID: 0000-0003-4652-9201) 
%was partially supported by the project Geometric and harmonic analysis with applications, funded by EU Horizon 2020 under the Marie Curie Grant Agreement No. 777822. 
%Keywords: \emph{measurable cocycle, numerical invariant, rigidity, boundary map}. MSC classes: 57T10, 53C35, 22E40, 22D40

\begin{abstract}
We introduce the notion of measurable bounded cohomology for measured groupoids, extending continuous bounded cohomology of locally compact groups.
We show that the measurable bounded cohomology of the semidirect groupoid associated to a measure class preserving action of a locally compact group $G$ on a standard Borel space is isomorphic to the continuous bounded cohomology of $G$ with twisted coefficients. 
We also prove the invariance of measurable bounded cohomology under similarity. As an application, we compare the bounded cohomology of (weakly) orbit equivalent actions and of measure equivalent groups. In this way we recover an isomorphism in bounded cohomology similar to one proved by Monod and Shalom. 

Other relevant consequences are related to the cohomological vanishing for actions of the Thompson group $F$, of higher rank lattices and of lattices in products of locally compact groups. 

We obtain a variant of the Eckmann-Shapiro isomorphism for transitive actions. In the case of a higher rank simple Lie group, we show that the cohomology of the action is actually determined by the usual cohomology of a suitable lattice. 

For amenable groupoids, we prove that the measurable bounded cohomology is trivial. This generalizes previous results by Monod, Anantharaman-Delaroche and Renault, and Blank.
\end{abstract}
  
\maketitle

\section{Introduction} 

\subsection{Motivational background}

Given a finitely generated group $\Gamma$ acting in a measure (class) preserving way on a standard Borel probability space $(X,\mu)$, the goal of \emph{measured group theory} is to understand the interplay between the algebraic properties of $\Gamma$ and the dynamics of the $\Gamma$-action. 
Interesting information are encoded, for instance, in the macroscopic structure of the orbits.
Given two measure preserving essentially free actions $\Gamma \curvearrowright X$ and $\Lambda \curvearrowright Y$ of finitely generated groups $\Gamma$ and $\Lambda$ on standard Borel probability spaces $X$ and $Y$, respectively, we say that the actions are \emph{orbit equivalent} if there exists a measure preserving Borel map $T:X\rightarrow Y$ which restricts to a measurable isomorphism between two full measure subsets and it sends $\Gamma$-orbits to $\Lambda$-orbits (see Example \ref{example_orbit_equivalence}  for a more precise definition). If the action $\Gamma \curvearrowright X$ is fixed, the class of all actions that are orbit equivalent to the chosen one may substantially vary according to the algebraic properties of $\Gamma$. For instance, Ornstein and Weiss \cite{OW80} proved that any ergodic probability measure preserving action of an infinite amenable group is actually orbit equivalent to an ergodic action of $\mathbb{Z}$. On the other hand, Zimmer \cite{zimmer:annals} proved an orbit equivalence rigidity result for actions of higher rank lattices. More precisely, he showed that an orbit equivalence between ergodic essentially free actions of two lattices $\Gamma, \Lambda$ contained in center-free simple Lie groups $G,H$, respectively, where at least one of them has higher rank, implies the existence of an isomorphism between the ambient groups. 

The notion of orbit equivalence can be weakened by requiring that the defining condition holds only when we restrict to two positive measure Borel subsets. In this case we say that two actions are \emph{weakly (or stably) orbit equivalent} (see again Example \ref{example_orbit_equivalence}). As shown in Furman' survey \cite{furman}, two essentially free ergodic probability measure preserving actions of finitely generated groups $\Gamma$ and $\Lambda$ are weakly orbit equivalent if and only if there exists a standard Borel $(\Gamma \times \Lambda)$-space $(\Omega,m)$ admitting two finite measure Borel fundamental domains, one for the $\Gamma$-action and the other one for the $\Lambda$-action. The latter condition defines an equivalence relation between finitely generated groups called \emph{measure equivalence}. It might be natural to ask if there is a way to classify all the groups appearing in a fixed measure equivalence class. Furman \cite{Furman99} proved that any finitely generated group $\Lambda$ which is measure equivalent to a higher rank lattice $\Gamma<G$, must be actually commensurable to a lattice of $G$, up to a finite kernel. A similar result was later extended by Bader, Furman and Sauer \cite{sauer:articolo} to the case of real hyperbolic lattices by introducing a technical integrability condition. Talking about measure equivalence rigidity theorems, they are worth mentioning the result proved by Monod and Shalom \cite{MonShal} for products of hyperbolic-like lattices, the paper by Kida \cite{kida1} about mapping class groups and the work by Alvarez and Gaboriau \cite{alvarez:gaboriau} about free products. 

An alternative point of view in the study of the dynamics of probability measure preserving actions comes from the theory of \emph{Borel equivalence relations}.
Given a measure preserving action of a finitely generated group $\Gamma$ on a probability space $(X,\mu)$, we denote by $\calR_\Gamma \subset X \times X$ the \emph{orbital equivalence relation} whose classes are exactly the $\Gamma$-orbits. An orbit equivalence between two actions $\Gamma \curvearrowright X$ and $\Lambda \curvearrowright Y$ implies the existence of a measured isomorphism between the associated orbital equivalence relations $\calR_\Gamma$ and $\calR_\Lambda$. This simple idea suggests that one may translate the dynamical properties of the action in terms of the associated orbital equivalence relation. In this spirit we naturally land in the world of countable standard Borel equivalence relations, namely those ones with countable classes and defined on a standard Borel probability space.
Feldmann and Moore \cite{feldman:moore} proved that any such relation $\calR$ is actually the orbital equivalence relation of some countable group. 
They also introduced a cohomology theory $\mathrm{H}^\bullet(\calR,T)$ with coefficients in an Abelian Polish group $T$. When the cohomological degree is one, we can actually extend the definition to any topological group $G$. Since a countable standard Borel relation is actually an orbital equivalence relation of some action $\Gamma \curvearrowright X$, we can lift $1$-cocycles to Borel maps $c:\Gamma \times X \rightarrow G$ satisfying $c(\gamma_1\gamma_2,x)=c(\gamma_1,\gamma_2x)c(\gamma_2,x)$ for every $\gamma_1,\gamma_2 \in \Gamma$ and almost every $x \in X$. We denote by $\mathrm{H}^1(\Gamma \curvearrowright X,G)$ the quotient of $1$-cocycles modulo cohomology. 

An explicit computation of $\mathrm{H}^1(\Gamma \curvearrowright X,G)$ may reveal quite harsh. For this reason it can be helpful to restrict the attention to specific families of groups $\Gamma$ and $G$. 
For instance, when $\Gamma$ is a higher rank lattice, $G$ is algebraic and $X$ is a standard ergodic probability space, Zimmer superrigidity theorem \cite{zimmer:annals} implies that the subset $\mathrm{H}^1_{ZD}(\Gamma \curvearrowright X,G)$ of Zariski dense classes coincides with the Zariski dense character variety $\chi_{ZD}(\Gamma,G)$. Inspired by both the functorial approach to bounded cohomology defined by Burger and Monod \cite{burger2:articolo} and the techniques used by Bader, Furman and Sauer \cite{sauer:articolo}, the authors \cite{sarti:savini:3}, together with Moraschini \cite{savini2020,moraschini:savini,moraschini:savini:2}, have recently introduced the theory of pull back along Borel $1$-cocycles in bounded cohomology, in the attempt to study more systematically the cohomology $\mathrm{H}^1(\Gamma \curvearrowright X,G)$. 
When $G$ is suitably chosen among semisimple Lie groups, they were able to exploit bounded cohomology to introduce a notion of \emph{maximality} for Borel $1$-cocycles. For such classes of cocycles on rank-one lattices a Zimmer Superrigdity holds.

When $G$ is Hermitian, there exists a preferred generating class in $\Hcb^2(G,\mathbb{R})$, called bounded K\"{a}hler class. The latter can be pulled back along a Borel $1$-cocycle $c:\Gamma \times X \rightarrow G$. In this way one can introduce the notion of \emph{parametrized K\"{a}hler class}. If $G$ is not of tube type, this is a complete invariant for the cohomology class of a Zariski dense cocycle \cite{sarti:savini:2}. Therefore one has an injection $K_X:\mathrm{H}^1_{ZD}(\Gamma \curvearrowright X,G) \rightarrow \Hb^2(\Gamma,\mathrm{L}^\infty(X,\mathbb{R}))$ of the Zariski dense cohomology of $\Gamma \curvearrowright X$ in the bounded cohomology of $\Gamma$ with twisted coefficients.

Since the theory of continuous bounded cohomology was originally developed to handle continuous coefficient modules and continuous representations, using it to study Borel 1-cocycles, whose nature is inherently measurable, appeared somewhat artificial to us. Our motivation behind this work stems from the conviction that any bounded cohomology theory applied within the framework of measured group theory should fully respect the measurable structure, without imposing any continuity requirements.
%The machinery developed to study cocycles via bounded cohomology is a natural adaption of techniques originally introduced for representations, and it is based on continuous bounded cohomology of groups.
%However, Borel 1-cocycles have their natural \emph{habitat} in the measurable world.
%It seems natural to ask if there exists a more natural framework to develop all this theory, namely if a measurable version of bounded cohomology can be defined in the measurable context and then employed in measured group theory.
%This fact,  together with the fruitful approach to cocycles via bounded cohomology, is the essential motivation behind this paper.

\subsection{Measurable cohomology of measured groupoids} The purpose of this paper is to set the foundations of a new cohomology theory in the framework of measured group theory,
exploiting the powerful tools coming from bounded cohomology. The starting point of our investigation relies on the possibility to view both group actions and equivalence relations as particular instances of the more general definition of \emph{groupoid}. A groupoid is a small category where all the morphisms are invertible. More roughly, a groupoid $\calG$ can be viewed as a weakened version of a group, where only some pairs of elements are composable (see Definition \ref{def_groupoid}). Additionally the set of units, usually denoted by $\calG^{(0)}$, does not boil down only to the neutral element, but it may be larger and there exist canonical maps $s,t:\calG \rightarrow \calG^{(0)}$, called \emph{source} and \emph{target} (see Definitions \ref{def_groupoid} and \ref{def_units}). 

Any equivalence relation $\calR$ defined on a subset $X$ can be naturally endowed with a groupoid structure. In this case the set of units coincides with $X$ and the maps $s,t$ are the restrictions to $\calR$ of the projections on the factors for $X \times X$ (Example \ref{example_groupoid_relation}). Another crucial example is given by a group $G$ acting on a set $X$. In this case the product $G \times X$ admits a structure of groupoid whose space of units is $X$ (Example \ref{example_groupoids_action}). The source map is simply the projection on $X$, whereas the target map coincides with the action map. This groupoid is called \emph{semidirect product groupoid} and it is usually denoted by $G \ltimes X$. If we assume that $(X,\mu)$ is actually a Lebesgue $G$-space, to keep track of the dynamical information, one needs to introduce a measured structure on the associated semidirect groupoid. More generally, a \emph{measured groupoid} is a pair $(\calG,C)$, where $\calG$ is a Borel groupoid and $C$ is a measure class on $\calG$ containing a \emph{symmetric} and \emph{quasi-invariant} probability measure (Definition \ref{definition_measured_groupoid}). In the case of the semidirect groupoid $\calG=G \ltimes X$ of a locally compact group $G$ acting in a measure preserving way, we can simply choose the class of the product between the Haar measure on $G$ and $\mu$. 

In the past decades, several cohomology theories of (measured) groupoids have been introduced. We already mentioned the cohomology theory of countable standard relations defined by Feldman and Moore \cite{feldman:moore}. Before that, Westman \cite{westman69} exploited the measured structure to introduce a cohomology theory via the resolution of measurable functions on fiber products of the starting groupoid. Anantharaman-Delaroche \cite{delaroche05} studied Westman's cohomology and obtained a vanishing result in degree one for measured groupoids with property $(T)$. Kida \cite{kida} showed that Westman's cohomology vanishes for countable ergodic probability measure preserving treeable relations when the cohomological degree is greater than or equal to $2$. 

In this paper we define the \emph{measurable bounded cohomology} of a measured groupoid $(\calG,C)$. We first need to understand which spaces are allowed to be chosen as coefficients. A \emph{measurable coefficient $\calG$-module} $E$ is the dual of a separable Banach space endowed with an isometric $\calG$-action whose contragradient representation has measurable orbital maps (see Definition \ref{definition_measurable_coefficient_module}).  With this choice of coefficients, we can consider the fiber product $\calG^{(\bullet)}$ with respect to the target map $t$, namely the space of tuples having the same target. Our cohomology theory is defined by taking the $\calG$-invariants of the resolution of essentially bounded weak$^*$ measurable functions $\Linfw(\calG^{(\bullet)},E)$, together with the usual homogeneous coboundary operator (Definition \ref{definition_measurable_bounded_cohomology}). As it happens for groups, the cohomology of such complex, denoted by $\Hmb^{\bullet}(\calG,E)$, inherits a natural seminormed structure induced by the essential supremum. When $\calG$ is actually a locally compact group, we obtain exactly the usual continuous bounded cohomology (Remark \ref{remark_bounded_cohomology_group}). 

The first important point of our cohomology theory is that it is invariant in the isomorphism class of the given groupoid (it is even more, being invariant in the similarity class, see Definition \ref{definition_homomorphism_homotopy_measured_groupoid}). Recall that an orbit equivalence between two essentially free actions $\Gamma \curvearrowright X$ and $\Lambda \curvearrowright Y$ on standard Borel probability spaces induces an isomorphism between the associated semidirect groupoids. As a consequence, we get that the measurable bounded cohomology groups of semidirect groupoids of orbit equivalent actions are isomorphic (Corollary \ref{corollary_orbit_equivalence}). In case of ergodic actions of countable groups, we can actually extend the result to weakly orbit equivalent actions. In other words, we can say that weakly isomorphic $\mathrm{II}_1$-relations share the same measurable bounded cohomology (Remark \ref{remark_ergodic_relation_weak_invariant}).  

Moving on in our investigation we obtain an exponential law which allows us to relate the measurable bounded cohomology of a semidirect groupoid $\calG=G \ltimes X$ with the usual continuous bounded cohomology of $G$ with twisted coefficients. % A technical assumption required for $X$ is its \emph{regularity}, namely the continuity of the induced $G$-action on the function space $\mathrm{L}^1(X)$ (Definition \ref{definition_regular_space}). 

\begin{intro_thm}\label{theorem_expo}
Let $G$ be a locally compact second countable group, let $X$ be a Lebesgue $G$-space and denote by $\calG=G\ltimes X$ the associated measured semidirect groupoid.
If $E$ is a measurable coefficient $\calG$-module we have canonical isometric isomorphisms
\begin{equation*}
\Hmb^{\bullet}(\calG,E)\cong \Hcb^{\bullet}(G,\Linfw(X,E)) \ ,
\end{equation*}
where the right-hand side denotes the continuous bounded cohomology of $G$ with coefficient in $\Linfw(X,E)$.
\end{intro_thm}

Notice that the $G$-action on $\Linfw(X,E)$ is induced by the structure of coefficient $\calG$-module of $E$. The usual exponential law for essentially bounded weak$^*$ measurable functions and the possibility to compute the twisted continuous bounded cohomology of $G$ via the $G$-invariants of the resolution $\Linfw(G^\bullet,\Linfw(X,E))$ enable us to write an explicit isomorphism at the level of cochains. 

Given a discrete countable group $\Gamma$ acting on a standard Borel probability space $X$, we introduce a \emph{comparison map} for the associated semidirect groupoid $\mathcal{G}=\Gamma \ltimes X$ (Definition \ref{definition comparison map}). For a separable Banach module $E$ with trivial $\calG$-action, the comparison map takes values in Westman's cohomology $\mathrm{H}_{\textup{m}}^\bullet(\Gamma \ltimes X,E)$ \cite{westman69}. Our Theorem \ref{theorem_expo} and Westman's exponential law \cite[Theorem 1.0]{westman:71} show that the comparison map for $\Gamma \ltimes X$ is actually conjugated to the composition of the usual comparison map $\Hb^\bullet(\Gamma,\Linf(X,E)) \rightarrow \mathrm{H}^\bullet(\Gamma,\Linf(X,E))$ with the map induced by the change of coefficients $\Linf(X,E) \rightarrow \mathrm{L}^0(X,E)$ (Proposition \ref{proposition comparison map}). 

The stability of bounded cohomology of semidirect groupoids of orbit equivalent actions, together with Theorem \ref{theorem_expo}, leads to the following 

\begin{intro_thm}\label{theorem_expo_orbit_equivalence}
Let $\Gamma,\Lambda$ be two countable groups and consider $\Gamma \curvearrowright X$ and $\Lambda \curvearrowright Y$ two essentially free measure preserving actions on standard Borel probability spaces. Denote by $\calG\coloneqq \Gamma\ltimes X$ and by $\calH\coloneqq \Lambda\ltimes Y$. 
If $T:X\rightarrow Y$ is a measurable isomorphism defining an orbit equivalence between $\Gamma \curvearrowright X$ and $\Lambda \curvearrowright Y$, then for every measurable coefficient $\calH$-module $E$ we have isometric isomorphisms
$$
\Hb^{\bullet}(\Lambda,\Linfw(Y,E)) \cong \Hb^{\bullet}(\Gamma,\Linfw(X,E)) \ ,
$$
induced by the cochain maps
\begin{gather*}
\mathcal{J}^{\bullet+1}:\Linfw(\Lambda^{\bullet+1}, \Linfw(Y,E))^{\Lambda} \rightarrow  \Linfw(\Gamma^{\bullet+1}, \Linfw(X,E))^{\Gamma}\,\\
 \mathcal{J}^{\bullet+1}\theta (\gamma_1,\ldots,\gamma_{\bullet+1})(x)= 
\theta (\sigma_T(\gamma_1,\gamma_1^{-1}x),\ldots,\sigma_T(\gamma_{\bullet+1},\gamma_{\bullet+1}^{-1}x))(T(x))
\end{gather*}
where $\gamma_1,\ldots,\gamma_{\bullet+1}$ in $\Gamma$ and $x\in X$ and $\sigma_T:\Gamma\times X\rightarrow \Lambda$ is the cocycle associated to the orbit equivalence.

If the actions are also ergodic but only weakly orbit equivalent, then we have similar isomorphisms
$$
\Hb^{\bullet}(\Lambda,\Linfw(Y,E)) \cong \Hb^{\bullet}(\Gamma,\Linfw(X,E)) \ .
$$
\end{intro_thm}

Just to lighten the notation, in the statement above we dropped the subscript ``c" for ``continuous", since the involved groups are endowed with the discrete topology. 

Our exponential law can be applied also to \emph{factors}: given two different actions $\Gamma \curvearrowright (X,\mu_X), \Gamma \curvearrowright (Y,\mu_Y)$, we say that $Y$ is a factor of $X$ is there exists a $\Gamma$-equivariant measure-preserving Borel map $T:X \rightarrow Y$. The function $T$ induces a morphism at the level of semidirect groupoids and we show that the associated map in bounded cohomology must be injective when the degree is equal to $2$, for any Hilbert coefficient module (Proposition \ref{prop_factor}). 

Given a locally compact group $G$, if we consider the semidirect groupoid $G\ltimes G/H$ of a transitive continuous action, where $H<G$ is a closed subgroup, 
we are able to establish a sort of Eckmann-Shapiro isomorphism between the measurable bounded cohomology of $G \ltimes G/H$ and the continuous bounded cohomology of $H$. More precisely we have the following

\begin{intro_prop}\label{proposition_homogeneous_semidirect_product}
Let $G$ be a locally compact group and let $H<G$ be a closed subgroup. Let $E$ be a coefficient $G$-module in the sense of Monod. If we endow $E$ with the structure of measurable coefficient $G \ltimes G/H$-module coming from the $G$-action, then we have isometric isomorphisms
$$
\Hmb^\bullet(G \ltimes G/H,E) \cong \Hcb^{\bullet}(H,E) \ .
$$
In particular, when $H$ is amenable the cohomology vanishes identically when the degree is greater than or equal to one. 
\end{intro_prop}

When $G$ is a higher rank simple Lie group and $(X,\mu)$ is a standard Borel probability $G$-space, Proposition \ref{proposition_homogeneous_semidirect_product}, together with the Stuck-Zimmer theorem \cite[Theorem 2.1]{stuck:zimmer}, implies that the bounded cohomology of the semidirect groupoid $G \ltimes X$ boils down the usual bounded cohomology of a suitable lattice $\Gamma< G$ (Corollary \ref{cor_stuck_zimmer}). 

Another important consequence of the exponential law is the vanishing of the degree $2$ measurable bounded cohomology for the semidirect groupoid of either the Thompson group $F$ (Corollary \ref{corollary_thompson}), a higher rank lattice (Proposition \ref{proposition_irreducible_higher_rank}) or an irreducible lattice in a product (Proposition \ref{proposition_lattice_products}). The latter results are deeply related to the vanishing of the parametrized class introduced by the authors \cite{sarti:savini:2} (Remark \ref{remark_vanishing_parametrized}). 

\subsection{Amenability of measured groupoids} The first notion of amenability for groups dates back to Von Neumann \cite{Neu29}. A locally compact group $G$ is \emph{amenable} if it admits a positive norm-one $G$-invariant mean defined on $\Linf(G,\mathbb{R})$. Amenability can be alternatively defined in terms of a fixed point property. More precisely, a group $G$ is amenable if, for any continuous linear $G$-action on a locally convex topological vector space $E$, any compact convex $G$-invariant subset of $E$ admits a fixed point (see Pier's book \cite[Section 1.4, Theorem 5.4]{pier}). Amenable groups are \emph{boundedly acyclic}, in the sense that their continuous bounded cohomology vanishes identically for every coefficient module \cite[Corollary 7.5.11]{monod:libro} and any positive degree. 
It is worth mentioning that explicit computations of bounded cohomology are rare, and the class of amenable groups is very significant since for it such invariant is fully known. 

Generalizing the fixed point property, Zimmer \cite[Chapter 4]{zimmer:libro} extended the notion of amenability to Lebesgue $G$-spaces. A Lebesgue $G$-space $X$ is \emph{amenable} if, for any linear action on the dual $E^*$ of separable Banach space $E$ induced by a Borel $1$-cocycle $G \times X \rightarrow \mathrm{Isom}(E)$, any invariant Borel field of convex weak$^*$ compact subsets of the closed unit ball in $E^*$ admits an invariant section. Later, Adams, Elliot and Giordano \cite{AEG} proved that the fixed point property is equivalent to the existence of a $G$-equivariant conditional expectation $\Linf(G \times X,\mathbb{R}) \rightarrow \Linf(X,\mathbb{R})$. For an amenable standard Borel $G$-space $X$ and any coefficient $G$-module $E$, Burger and Monod \cite[Theorem 1]{burger2:articolo} proved that the $G$-module of essentially bounded functions $\Linfw(X,E)$ is \emph{relatively injective}. As a consequence, the continuous bounded cohomology $\Hcb^\bullet(G,\Linfw(X,E))$ vanishes when the degree is greater than or equal to one \cite[Proposition 7.4.1]{monod:libro}. 

Similarly to the case of groups, the notion of amenability for groupoids can be given in different, but equivalent, ways.
We say that a measured groupoid $(\calG,C)$ is \emph{amenable} if it admits an invariant conditional expectation $\Linf(\calG) \rightarrow \Linf(\calG^{(0)})$ \cite[Definition 2.7]{delaroche:renault:01}. In the particular case of semidirect groupoids, the previous definition is exactly the one given by Adams, Elliot and Giordano. We show that any amenable measured groupoids is boundedly acyclic, generalizing the bounded acyclicity of amenable locally compact groups.

\begin{intro_thm}\label{theorem_amenability}
Let $\calG$ be an amenable measured groupoid and $E$ a measurable coefficient $\calG$-module. Then 
$$\Hmb^{\bullet}(\calG,E)\cong 0$$
whenever the degree is greater than or equal to one. 
\end{intro_thm}

The key point in the proof is that we can apply Hahn disintegration theorem \cite[Theorem 2.1]{Hahn} to realize the space $\Linfw(\calG^{(\bullet)},E)$ as a suitable fiber Bochner space over the units $\calG^{(0)}$ (Equation \ref{equation_disintegration_isomorphism}). The disintegration applies also the the conditional expectation and we get back an \emph{invariant system of means} (Definition \ref{definition_invariant_system_of_means}). With these ingredients, the proof is similar to the one in Frigerio's book \cite[Theorem 3.6]{miolibro}: by means of such invariant system, we construct a contracting homotopy for the complex $(\Linfw(\calG^{(\bullet)},E)^\calG,\delta^\bullet)$.

It is worth mentioning that Blank \cite{blank} already proved the bounded acyclicity for discretised amenable groupoids. Similarly, but using arguments not coming from homological algebra, Anantharaman-Delaroche \cite[Proposition 3.10, Proposition 3.11, Proposition 3.12]{delaroche:91} proved the vanishing of the first bounded cohomology group in the amenable case. In those cases, one can actually consider a more generally family of coefficients, namely \emph{bundles of coefficient modules}. 

We conclude this introduction noticing that Theorem \ref{theorem_amenability} is a generalization of several results that can be proved by the exponential law. Using the latter, one can prove that an amenable $\mathrm{II}_1$-relation is boundedly acyclic (Proposition \ref{proposition_ergodic_amenable_relation}). A similar statement holds for the semidirect groupoid of an amenable action (Proposition \ref{proposition_amenable_space}), in virtue of the relative injectivity proved by Burger and Monod. 

\vspace{5pt}

\paragraph{\textbf{Structure of the paper.}}
In Section \ref{section_Groupoids} we recall the necessary background about groupoids, with a particular focus on the measurable case. 
In Section \ref{section_Measurable coefficient modules} we introduce measurable coefficient modules, which are the natural coefficients for our cohomology theory. The latter is finally defined in Section \ref{section_Bounded cohomology of groupoids}. 
After a brief description of measurable bundles of Banach spaces (Section \ref{section_bundles}), we give the definition of measurable bounded cohomology (Section \ref{section_Measurable bounded cohomology}), we prove the invariance under similarity (Section \ref{section_Functoriality}) and we list some consequences. 
Then we move to the case of semidirect groupoids. Section \ref{section_The_exponential_law} is devoted to relate the continuous bounded cohomology of a group acting on a standard Borel space and the measurable one of the associated semidirect groupoid (Theorem \ref{theorem_expo}).
This result can be declined in different contexts, and in Section \ref{section_Main consequences of the exponential law} we describe some relevant applications. For instance, we consider the case of orbit equivalence and of measure equivalence, and we also deduce a vanishing result for the measurable bounded $2$-cohomology of an action of either the Thompson group, a higher rank lattice or a lattice in a product. We obtain a variant of Eckmann-Shapiro induction and we apply it to the case of higher rank actions that are not essentially free. 
Finally, in Section \ref{section_Amenability}, we prove the vanishing of measurable bounded cohomology for amenable groupoids (Theorem \ref{theorem_amenability}).
\vspace{5pt}

\paragraph{\textbf{Acknowledgements.}} The authors wish to thank Michelle Bucher and Tobias Hartnick for the interest they showed in this paper.

\section{Groupoids}\label{section_Groupoids}

In this preliminary section we are going to recall the necessary background about groupoids. We start giving the basic definitions and describing some examples. Then we move to the framework of measured groupoids, which are the main objects studied in this paper. We refer to the book by Muhly \cite{muhly} and to the one by Anantharaman-Delaroche and Renault \cite{delaroche:renault}.

\begin{defn}\label{def_groupoid}
A \emph{groupoid} is the datum of a set $\calG$, a unitary operation $^{-1}:\calG\rightarrow \calG$ and a semi-definite binary operation $$\calG\times \calG\rightarrow \calG\,,\;\;\;(g,h)\mapsto gh$$ satisfying the following conditions:
\begin{itemize}
\item[(i)] for every $g,h,k\in \calG$ such that $gh$ and $hk$ are defined, then 
also $(gh)k$ and $g(hk)$ are defined and $(gh)k=g(hk)$.
\item[(ii)] for every $g\in \calG$, then $gg^{-1}$ and $g^{-1}g$ are defined.
\item[(iii)] for every $g,h\in \calG$ such that $gh$ is defined then $g^{-1}gh=h$ and 
$ghh^{-1}=g$.
\end{itemize}
We denote by $\calG^{[2]}$ the set of composable pairs in $\calG^2$. The \emph{source map} is the function defined as $s(g)=g^{-1}g$, whereas the \emph{target map} is defined as $t(g)=gg^{-1}$.
\end{defn}

We notice that the images of the source and the target maps coincide. In fact, 
given $x \in \Ima(t)$, namely $x=gg^{-1}$ for some $x \in \calG$, 
then $s(g^{-1})=t(g)=x$, hence $x\in \Ima(s)$. The converse is analogous.

\begin{defn}\label{def_units}
The subset $$\calG^{(0)}\coloneqq\Ima(s)=\Ima(t)$$ of $\calG$ is the \emph{unit space} of $\calG$ and we say that $\calG$ is a \emph{groupoid over $\calG^{(0)}$}.
\end{defn}

If $\calG$ is a groupoid, then 
$(g,h)\in \calG^{[2]}$ if and only if $s(g)=t(h)$.
Indeed, if $(g,h)\in \calG^{[2]}$, then 
$$t(h)h=hh^{-1}h=h=g^{-1}gh=s(g)h$$
and hence $t(h)=s(g)$.
Viceversa, if $$s(g)=g^{-1}g=hh^{-1}=t(h)\,,$$ 
then by conditions (i) and (ii)
$$(g,g^{-1}g) \in\calG^{[2]}\,,\;\;\;  (hh^{-1},h)\in \calG^{[2]}\, \Rightarrow (g,h)\in \calG^{[2]}\,.$$
This basic fact is a useful criteria to establish when a pair is composable or not.

\begin{es}\label{example_groupoid_group}
Every \emph{group} $G$ is a groupoid, where $G^{[2]}=G^2$, 
$G^{(0)}=\{e_{G}\}$ and $t(g)=s(g)=e_G$ for every $g\in G$.
\end{es}

\begin{es}	\label{example_groupoids_action}
Given a group $H$ acting on a set $X$ from the left, we define the \emph{semidirect product groupoid} (or \emph{left action groupoid}) $H\ltimes X$ as follows:
the unit space is $(H\ltimes X)^{(0)}\coloneqq \{e_H\}\times X$ and the set of composable pairs is
$$(H\ltimes X)^{[2]}=\{((h,y),(g,x))\in (H\times X)^2\, | \,  gx=y\}\,.$$
The composition is given by 
$$(h,gx)(g,x)=(hg,x)\,,$$
and the inverse is defined as $$(g,x)^{-1}=(g^{-1},gx)\,.$$
The target map is $t(g,x)=(e_H,gx)$ and the source map is $s(g,x)=(e_H,x)$.

 \end{es}

\begin{es}\label{example_groupoids_action2}
Let $\calG$ be a groupoid and let $S$ be a set with a surjection $t_S:S\rightarrow \calG^{(0)}$.  If we define 
 $$\calG*_{t_S} S\coloneqq \{ (g,s)\in \calG\times S\, | \, s_{\calG}(g)=t_{S}(s)  \}\,,$$
where $s_{\calG}$ is the source on $\calG$, we say that $\calG$ \emph{acts on the left} on $S$ 
if there exists a map $$\calG*_{t_S} S\rightarrow S\,,\;\;\;\; (g,s)\mapsto gs $$
 satisfying the following conditions:
 \begin{itemize}
 \item $(gh)s=g(hs)$, whenever $(gh,s)\in \calG*_{t_S} S$ and $(g,hs) \in \calG*_{t_S} S $;
 \item $t_S(gs)=t_{\calG}(g)$, with $t_{\calG}$ target on $\calG$, whenever $(g,s) \in \calG*_{t_S} S$;
 \item $gg^{-1}s =g^{-1} gs=s$, whenever $(g,s)\in \calG*_{t_S} S$.
 \end{itemize}
Such a left action gives rise to a groupoid denoted by $\calG\ltimes S$ and called \emph{semidirect groupoid}, where 
 $$(\calG\ltimes S)^{[2]}=\{((g_1,s_1),(g_2,s_2))\in ( \calG*_{t_S} S )^2\, |\, g_2 s_2=s_1 \}\,, $$
and the product and the inverse map are as in Example \ref{example_groupoids_action}. The source map is given by $(g,s)\mapsto s$
and the target $t: \calG\ltimes S \rightarrow S$ is $(g,s)\mapsto g s$. When $\calG$ is a group acting on $S$, we recover the semidirect groupoid introduced in Example \ref{example_groupoids_action}.

When $S=\calG$ and $t_S=t_{\calG}$ is the target map of $\calG$, it follows that
$\calG*\calG=\calG^{[2]}$ and the action boils down to the usual composition on $\calG$.
More generally, if $\calG^{(\bullet)}$ denotes the set of $\bullet$-tuples of elements $g_1,\ldots,g_{\bullet}\in \calG$ sharing all the same target, we have a natural $\calG$-action on $\calG^{(\bullet)}$ defined as 
$$\calG*\calG^{(\bullet)} \rightarrow \calG^{(\bullet)}\,,\;\;\; (g,(g_1,\ldots,g_{\bullet}))\mapsto (gg_1,\ldots,gg_{\bullet})\,.$$
%
%Any groupoid $\calG$ can be viewed as a semidirect groupoid over its unit space, namely there is a natural set-theoretic identification $$\calG \sim \calG\ltimes \calG^{(0)}\,.$$
%sending an element $g\in \calG$ to the pair $(g,s(g))\in \calG\ltimes \calG^{(0)}$.
\end{es}
 
\begin{es}\label{example_groupoid_relation}
Let $X$ be a set and consider an \emph{equivalence relation} $\calR\subset X\times X$. 
The space $\calR$ has a natural structure of groupoid over $X$ whose set of composable pairs is 
$$\calR^{[2]}\coloneqq \{((x,y),(y,z))\in \calR^2\}\subset X\times X \,.$$
The composition map is defined as
$$(x,y)\circ (y,z)= (x,z)$$
 and the inverse of an element $(x,y)$ is $(y,x)$. 
There are two extreme cases, namely when either $\calR=X\times X$ or $\calR=\Delta\subset X\times X$. They are called \emph{trivial groupoid} and \emph{co-trivial groupoid}, respectively.

\end{es}

Group homomorphisms are generalized in this context by the following
\begin{defn}\label{definition_homomorphism_homotopy_similarity}
A \emph{homomorphism} between two groupoids $\calG$ and $\calH$ is a set-theoretic
map $f:\calG\rightarrow \calH$ such that $(g,h)\in \calG^{[2]} $ implies $(f(g),f(h))\in \calH^{[2]}$ and, in this case, it holds that
$$f(gh)=f(g)f(h)\,.$$
An invertible groupoid homomorphism is a groupoid \emph{isomorphism} and, whenever $\calG$ and $\calH$ are related by an isomorphism, we say that they are \emph{isomorphic} and we write $\calG\cong \calH$.

Given two homomorphisms $f_0,f_1:\calG\rightarrow \calH$, a \emph{similarity} from $f_0$ to $f_1$ is a map $h:\calG^{(0)}\rightarrow \calH$ such that for every $x\in \calG^{(0)}$ one has $s(h(x))=f_0(x)$, $t(h(x))=f_1(x)$ and for every $g\in \calG$ it holds
\begin{equation}\label{equation_homotopy}
h(t(g)) f_0(g)h(s(g))^{-1}= f_1(g)\,.
\end{equation}
If such a similarity exists, we write $f_0\simeq f_1$. Notice that the composability of the left-hand side of Equation \eqref{equation_homotopy} descends directly from the definition. 

Two groupoids $\calG,\calH$ are \emph{similar} if there exist homomorphisms 
$f_0:\calG\rightarrow \calH$ and $f_1:\calH\rightarrow \calG$ such that 
$$f_1\circ f_0\simeq \id_{\calG}\,,\;\;\; f_0\circ f_1\simeq \id_{\calH}\,.$$
Here the notation $\mathrm{id}$ refers to the identity map. 
\end{defn}

\begin{es}
Given a group action $H\curvearrowright X$ there exists a natural equivalence relation on $X\times X$, called \emph{orbital equivalence relation}, given by $(x,y) \in \mathcal{R}_H$ if and only if $y=hx$ for some $h \in H$. When the action is free, we have an isomorphism 
$$
\chi: H \ltimes X \rightarrow \calR_{H} \ , \ \ \chi(h,x):=(x,hx) \ , 
$$
between the semidirect groupoid $H\ltimes X$ of Example \ref{example_groupoids_action} and $\calR_H$ with the groupoid structure defined in Example \ref{example_groupoid_relation}.
\end{es}

\begin{oss}\label{remark_bijection_relations_groupoids}
The class of groupoids coming from equivalence relations is particularly interesting thanks to the following fact. Given a groupoid $\calG$ over $X$ with source and target maps denoted by $s$ and $t$, respectively, one can associate an equivalence relation $\calR_{\calG}\subset X\times X$ defined as
$$\calR_{\calG}\coloneqq \{(s(g),t(g))\, | \, g\in \calG\}\,.$$
If we identify two equivalence relations $\calR,\calR'$ whenever the associated groupoids (see Example \ref{example_groupoid_relation}) are isomorphic, it follows immediately that isomorphic groupoids $\calG,\calH$ produce equivalent relations $\calR_{\calG}\cong\calR_{\calH}$.

Notice that the equivalence relation $\calR_{\calG}$ forgets about the existence of non-trivial elements $g\in \calG$ with $s(g)=t(g)=x$, since in $\calR_{\calG}$ they are all identified with the diagonal element $(x,x)$. This phenomenon occurs for instance when $\calG$ is a group $G$. In fact, every element $g\in G$ stabilizes the unit $e_G$, namely $s(g)=t(g)=e_{G}$ and the equivalence relation $\calR_G$ is trivial.
This means that the above construction does not keep track of stabilizers of units.
For this reason, a well-defined correspondence between groupoids and equivalence relations  exists if we restrict to the subfamily of \emph{principal groupoids}, namely those such that the map $$\chi:\calG\rightarrow X\times X\,,\;\;\; \chi(g)=(s(g),t(g))$$
is injective.
In fact, for a given a set $X$ there exists a natural bijection 
\begin{equation}\label{equation_remark_bijection_relations_groupoids}
\left\{ \text{Equivalence relations on } X\right\}/_{\cong} \leftrightarrow\left\{ \text{Principal groupoids over } X\right\}/_{\cong}\,.
\end{equation}
and both sets are quotiented modulo isomorphism.

%As we will see in Remark \ref{remark_feldman_moore_theorem}, this  identification holds also in the Borel setting where, thanks to the work done by Feldman and Moore \cite{feldman:moore}, it defines a correspondence between countable standard Borel equivalence relations, principal groupoids with countable $t$-fibers and measure preserving actions of countable groups.
\end{oss}

Since we are interested in groupoids not only for their algebraic structure but also as a model for measure theoretic objects, such as group actions on measure spaces, we are going to endow them with a Borel structure.
\begin{defn}\label{definition_borel_groupoid}
A \emph{Borel groupoid} is a groupoid endowed with a $\sigma$-algebra such that the composition and the inverse are Borel maps (here $\calG^{[2]}$ is endowed with the Borel structure inherited by $\calG\times \calG$).

If $U\subset X$ is a Borel set, we denote by $\calG_{|U}$ the groupoid 
$$\calG_{|U}\coloneqq \{ g\in \calG\,|\, t(g),s(g)\in U\}\,$$
whose target and source maps are the restriction of the ones on $\calG$ and whose unit space is $U$.
\end{defn}

The notion of \emph{homomorphisms} and \emph{similarity} in the context of Borel groupoids are those given in Definition \ref{definition_homomorphism_homotopy_similarity}, with the further request that the involved maps must be Borel.
From now on, since we will work in the Borel context, we will tacitly assume that all maps are Borel. Furthermore, unless otherwise mentioned, all Borel spaces are assumed to be \emph{standard}.

\begin{es}
Let $G$ be a locally compact group acting on a standard Borel probability space $(X,\mu)$ in a measure class preserving way. We endow the semidirect groupoid $\calG=G \ltimes X$ with the Borel product structure. For another topological group $H$, a Borel morphism $f:\calG \rightarrow H$ must satisfy
$$
f(g_0g_1,x)=f(g_0,g_1x)f(g_1,x) \ ,
$$
for every $g_0,g_1 \in G$ and every $x \in X$. Such a map is called \emph{strict Borel cocycle}. Analogously, a similarity between two strict Borel cocycles $f_0,f_1:\calG \rightarrow H$ boils down to a \emph{strict cohomology} between $f_0$ and $f_1$, namely a measurable function $h:X \rightarrow H$ such that 
$$
h(gx)f_0(g,x)h(x)^{-1}=f_1(g,x) \ ,
$$
for every $g \in G$ and every $x \in X$. 
\end{es}

In what follows we will be mainly interested in (probability) measure preserving actions of a locally compact group on a (probability) measure space. In this context, the semidirect groupoid can be considered with a natural
$\sigma$-algebra and each fiber of the target map, being a copy of the group itself, can be endowed with the (renormalized) Haar measure. This produces a natural (probability) measure on the whole semidirect groupoid, obtained by ``spreading" the Haar measures on the fibers in a compatible way with the (probability) measure on the unit space (see Example \ref{example_groupoid_action_measurable}). This situation can be generalized to all Borel groupoids, and it requires the following
\begin{defn}\label{definition_borel_system}
A \emph{Borel Haar system of measures} for the target map $t:\calG\rightarrow \calG^{(0)}$ is a family $\rho=\{\rho^x\,|\,x\in \calG^{(0)}\}$ of $\sigma$-finite measures, where each $\rho^x$ is supported on the fiber $\calG^x\coloneqq t^{-1}(x)$, such that for every non-negative Borel map $f$ on $\calG$, the function 
$$x\mapsto \rho^x(f)$$ is Borel and such that $\rho$ is $\calG$-invariant, namely 
\begin{equation}\label{equation invariance system target}
g\rho^{s(g)}=\rho^{t(g)}\,.
\end{equation}
The latter condition can be expressed by the equality
$$\int_{\calG} f(g h)d \rho^{s(g)} (h)=\int_{\calG} f(h)d \rho^{t(g)}(h)\,.$$
\end{defn}

\begin{defn}\label{definition_invariant_borel_system}
Given a Borel groupoid $\calG$ endowed with a Borel Haar system of measures $\rho=(\rho^x)_{x \in \calG^{(0)}}$ for the target map $t$, a measure $\mu$ on $\calG^{(0)}$ is \emph{quasi-invariant with respect to $\rho$} if the convolution defined by 
\begin{equation} \label{equation_convolution_measures}
\mu*\rho(f)=\int_{\calG^{(0)}}\left( \int_{\calG}  f(g)d\rho^x(g)\right)d\mu(x)
\end{equation} 
 is equivalent to its direct image under the inverse map $g\mapsto g^{-1}$.
\end{defn}

The convolution $\mu \ast \rho$ naturally defines a measure on $\calG$ whose class $C$ is \emph{symmetric} and \emph{invariant}. A measure class $C$ on $\calG$ is \emph{symmetric} if any measure $\nu \in C$ is equivalent to the direct image $\nu^{-1}$ under the inverse map $g\rightarrow g^{-1}$. To better explain the notion of invariance, we recall that given a $\sigma$-finite measure $\nu$ on $\calG$, equivalent to a probability measure $\widetilde{\nu}$ with $t_*\widetilde{\nu}=\mu$, it can be \emph{disintegrated} with respect to the target $t:\calG \rightarrow \calG^{(0)}$ (see for instance Effros \cite[Lemma 4.4]{Eff66} or Hahn \cite[Theorem 2.1]{Hahn}). More precisely, there exists a Borel map from the unit space $\calG^{(0)}$ to the space of measures on $\calG$, sending $x \mapsto \nu^x$, where $\nu^x$ is supported on $\calG^x$, such that
\begin{equation}\label{equation_isom_disintegration}
\int_{\calG} f(g)d\nu(g)=\int_{\calG^{(0)}} \left(\int_{\calG} f(g)d\nu^x(g)\right) d\mu(x) \ ,
\end{equation}
for any $f \in \mathrm{L}^1(\calG)$. We call $(\nu^x)_{x \in \calG^{(0)}}$ the $t$-\emph{disintegration} of the measure $\nu$ (notice that \cite[Theorem 2.1]{Hahn} states the essential uniqueness of such disintegration). We say that the $t$-disintegration is \emph{quasi} $\calG$-\emph{invariant} if $g\nu^{s(g)}$ and $\nu^{t(g)}$ define the same measure class, for $\nu$-almost every $g \in \calG$. We say that $C$ is \emph{invariant} if it contains a probability measure whose $t$-disintegration is quasi $\calG$-invariant. 

The previous discussion leads to the following 
\begin{defn}\label{definition_measured_groupoid}
A \emph{measured groupoid} is a Borel groupoid $\calG$ endowed with a symmetric and invariant measure class $C$.
\end{defn}

The above considerations show that a Borel groupoid $\calG$ equipped with a Haar system $(\rho^x)_{x \in \calG^{(0)}}$ and quasi-invariant measure $\mu$ on $\calG^{(0)}$ is in fact a measured groupoid by taking the measure class of the convolution $\mu \ast \rho$. Conversely, given a measured groupoid $(\calG,C)$, Hahn \cite{Hahn} showed that we can write the measure class $C$ in an essentially unique way as the convolution defined by Equation \eqref{equation_convolution_measures} on an inessential contraction of $\calG$ (see Definition \ref{definition_inessential_contraction}). It is worth noticing that, as in the case of locally compact groups, the finiteness of a measure $\nu$ in $C$ conflicts with the invariance of its $t$-disintegration: if one requires that $\nu$ is a probability measure, then its disintegration is in general only quasi-invariant; on the other hand, if one needs an invariant disintegration, then $\nu$ is $\sigma$-finite but not a priori finite.  

Given a measure class $C$ defining a measured groupoid $\calG$, we will denote by $\nu$ a symmetric and quasi-invariant probability measure in $C$, by $(\nu^x)_{x \in \calG^{(0)}}$ the $t$-disintegration of $\nu$ and by $\mu$ the direct image of $\nu$ on the space of units $\calG^{(0)}$.
Finally, we will refer to a measured groupoid either as a pair $(\calG,C)$ or as a pair $(\calG,\nu)$, with $C$ and $\nu$ as above. Only in some specific cases, like the example below, we will prefer to consider the Borel Haar system on $\calG$ rather than the quasi-invariant probability measure $\nu$. 

\begin{es}\label{example_groupoid_action_measurable}
Let $G$ be a locally compact group with Haar measure $\rho$ and $(X,\mu)$ a Lebesgue $G$-space, that is a standard Borel measure space on which $G$ acts by preserving the measure class. The Haar measure $\rho^x=\rho$ on each $G^x=G$ defines a Borel Haar system of measures and $\mu$ is quasi-invariant with respect to $\rho^x$. This turns the semidirect groupoid $G\ltimes X$ into a measured groupoids.

For example, if $H<G$ is a closed subgroup, then the quotient $G/H$ admits a natural structure of Lebesgue $G$-space. As a consequence we can naturally view $G\ltimes G/H$ as a measured groupoid. 

More generally, consider the semidirect groupoid associated to the left action of a measured groupoid $\calG$ on a standard Borel probability space $(S,\tau)$. In this context $\calG \ltimes S$ is Borel if the action map of Example \ref{example_groupoids_action2} is so. If $\tau$ is quasi-invariant in the sense of 
Anantharaman-Delaroche and Renault \cite[Definition 3.1.1]{delaroche:renault}, that is it satisfies a condition similar to Equation \ref{equation_convolution_measures}, the semidirect groupoid $\calG \ltimes S$ can be endowed with a natural structure of measured groupoid. Since we will mainly work with the case $(S,\tau)=(\calG^{(0)},\mu)$, we omit any other detail and we refer the reader to \cite[Chapter 3]{delaroche:renault} for a complete discussion about measured semidirect groupoids.
\end{es}

In the measurable context everything is considered up to null-measure set. For this reason, it is often useful to restrict to a ``full-measure subgroupoid" of a given measured groupoid. This motivates the following 
\begin{defn}\label{definition_inessential_contraction}
Let $\calG$ be a measured groupoid with unit space $(X,\mu)$  and with measure class $C$. Let $U\subset X$ be a Borel subset such that $\calG_{|U}$ is $\nu$-conull for some (and hence any) $\nu\in C$. The \emph{inessential contraction} of $\calG$ to $U$ is the measured groupoid $\calG_{|U}$ endowed with the measure class $C_{|U}$ given by the restriction of (the elements of) $C$.
\end{defn}

In the sequel we want consider measurable maps defined on a measured groupoid identified up to null sets and their composition with groupoid homomorphisms. Such composition can easily give troubles: for instance the image of the Borel homomorphism can have null-measure in the target. In this case, we would attempt to evaluate an equivalence class of measurable functions on a null set, which is clearly meaningless.
For this reason, it is convenient to select a strict subfamily of the whole set of Borel homomorphisms of Definition \ref{definition_borel_groupoid} which respects the measure class.

\begin{defn}\label{definition_homomorphism_homotopy_measured_groupoid}
Let $\calG$ (respectively $\calH$) be a measured groupoid with unit space $(X,\mu_X)$ (respectively $(Y,\mu_Y)$) endowed with a measure $\nu$ (respectively $\tau$) representing the invariant symmetric measure class of $\calG$ (respectively of $\calH$).
A \emph{homomorphism} between $(\calG,\nu)$ and $(\calH,\tau)$ is a measurable map $f:\calG\rightarrow \calH$ such that $f_*\nu$ is absolutely continuous with respect to $\tau$, for almost every $(g,h)\in \calG^{[2]}$ one has $(f(g),f(h))\in \calH^{[2]}$ and in this case
$$f(gh)=f(g)f(h)\,.$$

A homomorphism $f:\calG \rightarrow \calH$ is an \emph{isomorphism} if there exists another morphism $k:\calH \rightarrow \calG$ which is an inverse of $f$ when we restrict both to two inessential contractions, one of $\calG$ and the other one of $\calH$. In this case we say that $(\calG,\nu)$ and $(\calH,\tau)$ are isomorphic and we write $\calG\cong \calH$.

A \emph{weak isomorphism} between $(\calG,\nu)$ and $(\calH,\tau)$ is an isomorphism $f:\calG_{|U}\rightarrow \calH_{|V}$ for some $U\subset X$ and $V\subset Y$ of positive measure, where the measures on $\calG_{|U}$ and on $\calH_{|V}$ are $\nu'(A)\coloneqq \nu(A\cap U) \cdot \nu(U)^{-1}$ and $\tau'(B)\coloneqq \tau(B\cap V) \cdot \tau(V)^{-1}$, respectively.

A \emph{similarity} between $f_0,f_1:\calG\rightarrow\calH$ is a measurable map $h:X\rightarrow \calH$ such that for $\mu_X$-almost every $x\in X$ one has $s(h(x))=f_0(x)$ and $t(h(x))=f_1(x)$ and 
for $\nu$-almost every $g\in \calG$ it holds
$$h(t(g)) f_0(g)h(s(g))^{-1}= f_1(g)\,.$$
If such a similarity exists, we still write $f_0\simeq f_1$.

Two measured groupoids $\calG,\calH$ are \emph{similar} if there exist homomorphisms 
$f_0:\calG\rightarrow \calH$ and $f_1:\calH\rightarrow \calG$ as above such that 
$$f_1\circ f_0\simeq \id_{\calG}\,,\;\;\; f_0\circ f_1\simeq \id_{\calH}\,,$$
where $\mathrm{id}$ refers to the identity map. 
\end{defn}

\begin{oss}\label{remark_homomorphism}
It is worth noticing that our homomorphisms may not send units to units, since the unit space can have null measure in the groupoid. 
Definition \ref{definition_homomorphism_homotopy_measured_groupoid} refers to \emph{a.e. homomorphisms} in Muhly's book \cite[Definition 4.6]{muhly}.
Instead of such morphisms, both Ramsay and Muhly work with Borel maps that restrict to homomorphisms on an inessential contraction and that satisfy the condition of absolute continuity of Definition \ref{definition_homomorphism_homotopy_measured_groupoid} at the level of unit spaces.
The first one call them simply \emph{homomorphisms}, whereas the second one refers to them as \emph{weak homomorphisms}.
Thanks to the work done by Ramsay, any a.e. homomorphism coincides almost everywhere with a weak homomorphism in his notation \cite[Theorem 5.1]{ramsay}. 
Although weak homomorphisms send almost all units to units,
some issues arise in their composition, since a priori it may not make sense. Ramsay circumvented this problem by considering \emph{similarity classes} of weak homomorphisms.
On the other hand, the composition of two homomorphisms in our setting is always defined.
For the purposes of this paper, we therefore prefer to consider such class of morphisms and to rely on Ramsay's result when it is necessary.
\end{oss}

\begin{es}\label{example_factor}
Consider two measure preserving actions $\Gamma \curvearrowright (X,\mu_X), \Gamma \curvearrowright (Y,\mu_Y)$ of a countable group $\Gamma$ on two standard Borel probability spaces. We say that $Y$ is a \emph{factor} of $X$, and we write $Y \leq X$, if there exists a Borel $\Gamma$-equivariant map $T:X \rightarrow Y$ such that $T_\ast \mu_X=\mu_Y$. As noticed by Kechris \cite[Chapter II.10]{Kec09}, a theorem by Rohklin guarantees that if $Y \leq X$ and $X$ is ergodic, then $X$ can be factorized as $Y \times Z$, where $(Z,\mu_Z)$ is some standard Borel probability space, and the action on $X$ is the skew-product of the one on $Y$ with a Borel cocycle $\Gamma \times Y \rightarrow \mathrm{Aut}(Z)$ in the measure preserving automorphisms of $Z$. 

If $Y \leq X$ and we set $\calG\coloneqq \Gamma \ltimes X, \calH\coloneqq\Gamma \ltimes Y$, both endowed with the measured product structure, then the map 
$$
f_T:\calG \rightarrow \calH \ , \ \ f_T(\gamma,x):=(\gamma,T(x)) 
$$
is a homomorphism of measured groupoids in the sense of Definition \ref{definition_homomorphism_homotopy_measured_groupoid}.
\end{es}

\begin{es}\label{example_orbit_equivalence}
Let $\Gamma\curvearrowright (X,\mu_X)$ and $\Lambda\curvearrowright (Y,\mu_Y)$ be measure preserving essentially free actions of countable groups on standard Borel probability spaces. 
An \emph{orbit equivalence} is a Borel map $T:X\rightarrow Y$ with $T_*\mu_{X}= \mu_{Y}$ restricting to a Borel isomorphism between conull Borel subsets $X'\subset X$ and $Y'\subset Y$ and such that 
$T(\Gamma  \cdot x \cap X')=(\Lambda\cdot  T (x)) \cap Y'$ for every $x\in X'$.

A \emph{weak orbit equivalence} is an orbit equivalence between positive measure subsets $U\subset X$ and $V\subset Y$ such that $T_* \mu_{U}=\mu_{V}$, where $\mu_{U}(T)\coloneqq \mu_{X} (T\cap U)\cdot \mu_{X}(U)^{-1}$ for every measurable $T\subset X$ and $\mu_{V}(S)\coloneqq \mu_{Y} (S\cap V)\cdot \mu_{Y}(V)^{-1}$ for every measurable $S\subset Y$.

An orbit equivalence between essentially free actions $\Gamma\curvearrowright (X,\mu_X)$ and $\Lambda\curvearrowright (Y,\mu_Y)$ corresponds to an isomorphism of two inessential contractions of the semidirect groupoids $\Gamma\ltimes X$ and $\Lambda\ltimes Y$. 
More precisely, one can define almost everywhere $$f_T:\Gamma\ltimes X\rightarrow \Lambda\ltimes Y\,,\;\;\; f_T(\gamma,x)\coloneqq (\sigma_{T} (\gamma,x), T(x))$$
where $\sigma_{T}: \Gamma\times X\rightarrow \Lambda$ is the orbit equivalence cocycle, namely it satisfies
$$T(\gamma  x)=\sigma_{T} (\gamma,x)T(x)$$
for every $\gamma\in \Gamma$ and almost every $x\in X$ (notice that $\sigma_{T}$ is well-defined since the $\Lambda$-action is essentially free).
Then $f_T$ is a groupoid isomorphism between two inessential contractions, whose inverse can be built using the inverse of $T$ and exploiting the essentially free $\Gamma$-action. 

Something completely analogous holds in the case of a weak orbit equivalence, but this time we only obtain a weak isomorphism between $\Gamma \ltimes X$ and $\Lambda \ltimes Y$. 

\end{es}

An interesting family of measured groupoids is given by groupoids with countable fibers.
\begin{defn}\label{definition_t_discrete_groupoid}
A measured groupoid $\calG$ is \emph{$t$-discrete} if $\calG^x:=t^{-1}(x)$ is countable for every $x\in X$.
\end{defn}
 
For a $t$-discrete groupoid $(\calG,C)$ over $(X,\mu)$, a representative $\nu$ of the symmetric and invariant measure class $C$ can be defined as
$$\nu(V)\coloneqq \int_X |V\cap \calG^x| d\mu(x)$$
for every Borel subset $V\subset \calG$.

\begin{es}
When $\Gamma$ is a discrete countable group (endowed with the counting measure) acting on a Lebesgue $\Gamma$-space $X$, the semidirect groupoid $\Gamma\ltimes X$ introduced in Example \ref{example_groupoids_action} is $t$-discrete.
\end{es}

\begin{es}\label{remark_feldman_moore_theorem}
If $X$ is a standard Borel space with Borel sigma algebra $\calB$ an equivalence relation $\calR\subset X\times X$ is called \emph{standard} if $\calR\subset  \calB\times \calB$, where $\calB\times \calB$ is the induced Borel structure on $X\times X$. 
Any such equivalence relation whose equivalence classes are countable is in fact a $t$-discrete groupoid (see Example \ref{example_groupoid_relation}).
Feldmann and Moore \cite[Theorem 1]{feldman:moore} showed that any standard equivalence relation on a space $X$ whose equivalence classes are countable comes from an action of a countable group on $X$ (which is not necessarily essentially free). An important family of equivalence relations in the Borel setting is the one of \emph{$\mathrm{II}_1$-relations}, namely countable standard equivalence relations coming from a \emph{measure class preserving} and \emph{ergodic} group action on a \emph{non-atomic} standard Borel space (see also \cite[Section 4.1]{furman}). Such relations are considered in Section \ref{section_Main consequences of the exponential law}, where we prove a vanishing result of the measurable bounded cohomology in the amenable case.

%Moreover, since the correspondence between principal groupoids and equivalence relations discussed in Remark \ref{remark_bijection_relations_groupoids} extends to the Borel setting, it follows that any isomorphism class of $t$-discrete groupoids over a standard Borel space $X$ can be represented by a semidirect groupoid of an action of a countable group on $X$. 
\end{es}

%The result by Feldman and Moore allows to call a $t$-discrete principal groupoid \emph{ergodic} if the semidirect groupoid in its equivalence class comes from an ergodic action. The notion of ergodicity can be actually extended to the whole family of measured groupoids. 

\begin{defn}\label{definition_ergodic_groupoid}
A measured groupoid is \emph{ergodic} if for every Borel subset $U\subset X$ we have that its saturation
$$\calG U\coloneqq \{ t(g)\,|\, s(g) \in U\}$$ 
is either null or conull in $X$.
\end{defn}

%\begin{oss}\label{remark_ergodic_relation}
%Thanks to the identification between principal groupoids and equivalence relations discussed in Remark \ref{remark_feldman_moore_theorem}, we notice that ergodic measured groupoids corresponds to ergodic equivalence relations in the sense of Feldmann and Moore \cite{feldman:moore}.
%\end{oss}

\section{Measurable coefficient modules}\label{section_Measurable coefficient modules}

The aim of this section is to introduce the coefficient modules for the cohomological functor that will be defined in Section \ref{section_Bounded cohomology of groupoids}. 
Precisely, we need to formalize how a measured groupoid acts (measurably) on a Banach space and to give a notion of equivariant map between such spaces. 
We also make a comparison between the continuous case considered by Monod \cite{monod:libro}, that will be used in Section \ref{section_The_exponential_law}
to relate measurable bounded cohomology to continuous bounded cohomology.

We start with the the following
\begin{defn}
A \emph{coefficient module} is a Banach space $E$ that is the dual of a separable Banach space $F$.
\end{defn}

We are interested into coefficient modules endowed with an isometric action of a measured groupoid. In this section, when a measured groupoid is mentioned, we will not explicitly refer neither to its symmetric and invariant measure class, nor to any of its representative (see Section \ref{section_Groupoids}). Nevertheless, the sentence ``almost every" should not generate any ambiguity and all the measures should be clear from the context when they are not specified.

\begin{defn}\label{definition_measurbale_G_module}
Let $\calG$ be a measured groupoid. A \emph{measurable} Banach $\calG$-module $E$ is a Banach space endowed with a left action by isometries of $\calG$, that is a function 
$$L:\calG\rightarrow \Isom(E)$$ satisfying
$L(gh)=L(g)L(h)$ for almost every pair $(g,h)\in \calG^{[2]}$
and such that the orbital map 
$$\calG \rightarrow E\,,\;\;\; v\mapsto L(g) v$$ is measurable for every $v\in E$. 
\end{defn}

In particular our measurability request concerns the action on the predual of our module. 
\begin{defn}\label{definition_measurable_coefficient_module}
Let $\calG$ be a measured groupoid over $X$.
A \emph{measurable coefficient $\calG$-module} is a pair $(E,L)$ where $E$ is the dual of a measurable $\calG$-module $E^{\flat}$ and
$$L:\calG\rightarrow \Isom(E)$$ is the contragradient representation of a given $\calG$-action $L^{\flat}$ on $E^\flat$, namely
$$\langle L(g) \lambda,v\rangle \coloneqq \langle \lambda,L^{\flat}(g^{-1}) v\rangle$$
for every $\lambda \in E$, $v\in E^{\flat}$ and $g\in\calG$.
\end{defn}

This is the natural transposition to measured groupoids of the notion of coefficient modules given by Monod \cite[Definition 1.2.1]{monod:libro}.

\begin{es}
Let $G$ be a topological group. A coefficient $G$-module in the sense of Monod is a measurable $G$-module if $G$ is viewed as a groupoid.
Under relatively weak assumptions on $G$, such as local compactness, the converse is also true (see \cite[Proposition 1.1.3]{monod:libro} and Remark \ref{remark_bounded_cohomology_group}).
\end{es}

\begin{es}\label{example_pull_back_via_groupoid_homomorphism}
Let  $\calG,\calH$ be two measured groupoids over $X$ and over $Y$, respectively. Consider $F$ a measurable coefficient $\calH$-module.
A homomorphism $f: \calG \rightarrow \calH$ induces a $\calG$-action on $F$, that is
$$f^*L:\calG\rightarrow \Isom(F)\,,\;\;\; f^*L(g)\coloneqq L(f(g))\,.$$
With a slight abuse of notation, in the sequel we denote by $F$ also the $\calG$-module described above.
\end{es}

A map between coefficient modules consists of a measurable family of linear maps parametrized with the units intertwining the actions.
Precisely, if $\calL(E,F)$ denotes the space of linear bounded operators endowed with the Borel $\sigma$-algebra given by the strong operator topology, we have the following
\begin{defn}\label{definition_morphism_coefficient_modules}
Given two measurable coefficient $\calG$-modules $(E,L_E)$ and $(F,L_F)$, a \emph{$\calG$-map} between them is a measurable map 
$$\phi:\calG^{(0)}\rightarrow \calL(E,F)\,,\;\;\; \phi(x)\coloneqq\phi_x$$ such that
\begin{equation}\label{equation_intertwine_actions}
L_F(g)\cdot ( \phi_{s(g)}(v))= \phi_{t(g)} (L_E(g)\cdot  (v))
\end{equation}
for almost every $g\in \calG$ and every $v\in E$.

A $\calG$-map is \emph{bounded} if the map $$\calG^{(0)}\rightarrow \matR\,,\;\;\; x\mapsto ||\phi_x||_{\calL} $$ is essentially bounded, where $\| \cdot \|_{\calL}$ is the operator norm. The \emph{norm} of $\phi$ is 
$$||\phi||_{\infty}\coloneqq \textup{ess sup}\{\; ||\phi_x||_{\calL} \ | \ x \in \calG^{(0)} \} \,.$$

Two measurable coefficient $\calG$-modules $(E,L_E)$ and $(F,L_F)$ are \emph{isomorphic} if they are isomorphic as Banach spaces and there exists a $\calG$-map $\phi:\calG^{(0)}\rightarrow \Isom(E,F)$ taking values in the space of isometries between $E$ and $F$. Such a $\calG$-map is an \emph{isomorphism} of measurable coefficient $\calG$-modules.
\end{defn}

\begin{oss}
In the above definition, an isomorphism corresponds to the existence of a measurable family of isometries that intertwine the $\calG$-actions. 
This is equivalent to the data of both an isomorphism of Banach spaces $T:E\rightarrow F$ and a $\calG$-map $\psi:\calG^{(0)} \rightarrow \Isom(E)$ such that 
$$L_F(g) \cdot T(\psi_{s(g)}(v))=T (\psi_{t(g)}(L_E(g)\cdot  v)) \ .$$ 
We notice that $\psi$ can be interpreted as a similarity (Definition \ref{definition_homomorphism_homotopy_measured_groupoid}) between $L_E$ and $L_F$ conjugated by $T$.
This form could reveal misleading for the reader because of the apparent lack of symmetry between $E$ and $F$ that, in our opinion, becomes more clear in Definition \ref{definition_morphism_coefficient_modules}.
\end{oss}

\begin{oss}\label{remark_actions_identified}
Since Equation \eqref{equation_intertwine_actions} holds for almost every $g \in G$, this suggests that an action $L: \calG\rightarrow \Isom (E)$ on a measurable coefficient module $E$ can be perturbed on a null subset of $\calG$, still getting an isomorphic $\calG$-action on $E$. This shows once more how in the measurable setting it is natural, and also necessary, to identify maps that coincide almost everywhere. From now on we therefore identify isomorphic actions.
\end{oss}

%\todo{Correggere il remark qui sotto: adesso parliamo di fibrati quindi riferiamoci a quella sezione. In generale, ovunque parliamo di fibrati, va rivisto alla luce della nuova sezione introdotta}

\begin{oss}\label{remark_bundles}
In the context of groupoids, instead of taking a single Banach space as above, it is often useful to consider the more general notion of measurable $\calG$-bundle of Banach spaces. Roughly speaking, the idea of measurable bundles is to attach a Banach space to each unit of a measured groupoid $\calG$ in a measurable way. We refer the reader to Section \ref{section_bundles} for more details about such definition. A notion of $\calG$-action is also given in this more general context \cite[Definition 4.1.1]{delaroche:renault}, and it seems natural to ask whether it coincides with Definition \ref{definition_measurbale_G_module}. If we consider the constant bundle, namely the family $\{E_x\}_{x\in \calG^{(0)}}$, where $E_x=E$ for some fixed Banach space $E$, a $\calG$-action is measurable in the sense of Anantharaman-Delaroche and Renault \cite[Definition 4.1.1]{delaroche:renault} if the map 
$$\calG\rightarrow E\,,\;\;\; g\mapsto L(g) \sigma(s(g))$$ is $\nu$-measurable
for every $\mu$-measurable $\sigma:\calG^{(0)}\rightarrow E$. Here $\nu$ is the quasi $\calG$-invariant probability representing $C$, $t_\ast \nu=\mu$ and $\nu$-measurability (respectively $\mu$-measurability) refers to the completion of the measure $\nu$ (respectively $\mu$).
First of all, we notice that such condition is equivalent to require that the function 
$$\calG\rightarrow E\,,\;\;\;g\mapsto L(g) v$$
is $\nu$-measurable for every $v\in E$. 
In fact, the non-trivial implication follows from the fact that for each $\mu$-measurable $\sigma$ there exists a sequence $(\sigma_n)$ of functions of the form $\sum \mathbbm{1}_{A_i}\cdot  v_i$, where $A_i\subset X$ are Borel subsets and $v_i\in E$ such that 
$ \sigma_n(x)\rightarrow\sigma(x)$
 for almost every $x\in X$ \cite[Chapter II, Section 5.1]{fell:doran}.
Hence $L(g)\sigma(s(g))$ coincides almost everywhere with the limit of $\nu$-measurable functions $L(g)\sigma_n(s(g))$, thus it is also $\nu$-measurable. 

One can prove even more, namely that if $E$ is a separable measurable coefficient module, any measurable action $\calG\rightarrow \Isom(E)$ in the sense of Anantharaman-Delaroche and Renault coincides with a measurable one (Definition \ref{definition_measurbale_G_module}) except on a $\nu$-negligible set.
This fact, together with Remark \ref{remark_actions_identified}, shows that for separable Banach spaces our definition coincides with the one by Anantharaman-Delaroche and Renault.
\end{oss}

In the following example we describe how to pullback the action of a groupoid on a measurable coefficient through similarities. 

\begin{es}\label{example_coefficient_map_induced_by_homotopy}
Let $\calG,\calH$ be two measured groupoids, $f_0,f_1:\calG\rightarrow \calH$ two homomorphisms and let $h:\calG^{(0)}\rightarrow \calH$ be a similarity from $f_0$ to $f_1$ (see Definition \ref{definition_measured_groupoid}). 
Thanks to Example \ref{example_pull_back_via_groupoid_homomorphism}, a measurable coefficient $\calH$-module $(F,L)$ has also a double structure of measurable coefficient $\calG$-module with the pullback actions induced by $f_0$ and $f_1$, respectively.
The similarity $h$ defines a $\calG$-map between coefficient modules $$H:\calG^{(0)}\rightarrow \Isom(f_0^*F,f_1^*F)\,,\;\;\; H(x)\coloneqq H_x$$
where
$$H_x: F\rightarrow F\,,\;\;\; H_x(v)\coloneqq L(h(x)) \cdot v$$
for every $x\in \calG^{(0)}$. We claim that $H$ intertwines the $\calG$-actions. In fact, for almost every $g\in \calG$ we have 
\begin{align*}
H_{t(g)} (f_0^*L(g)\cdot v)&= L(h(t(g)) \cdot (L(f_0(g))\cdot v)\\
&=L(h(t(g)) f_0(g))\cdot v\\
&= L(f_1(g)h(s(g))\cdot v= f_1^*L(g) \cdot (H_{s(g)} ( v))\,.
\end{align*}
%Moreover, denoting by $\widetilde{h}$ the map $x\mapsto h(x)^{-1}$, the inverse of $h_x$ is 
%$$\widetilde{h}_x:F\rightarrow F\,,\;\;\; \widetilde{h}_x(v)\coloneqq L(\widetilde{h}(x))\cdot v\,.$$
Moreover, for every $v\in E$, the measurability of $x\mapsto H_x(v)$ follows from the measurability of both $h$ and the $\calH$-action.
\end{es}

\section{Measurable bounded cohomology of measured groupoids}\label{section_Bounded cohomology of groupoids}

This section is divided into three parts. In the first part there is a short description of fiber Bochner spaces. We focus our attention on the particular cases of both $\mathrm{L}^1$ and $\Linf$. The importance of the latter one becomes evident in the second part of the section, where we introduce the \emph{measurable bounded cohomology} of a measured groupoid and we list its basic properties. The definition is inspired by Monod's approach  \cite{monod:libro} via the complex of weak$^*$ measurable essentially bounded functions. Using Hahn disintegration theorem \cite[Theorem 2.1]{Hahn}, that complex can be alternatively viewed as a complex of fiber $\Linf$-spaces over the units of the groupoid. It is worth noticing that when the groupoid is actually a locally compact group, one recovers its continuous bounded cohomology. The last part of the section is spent to prove that our construction is functorial, namely that similar groupoid homomorphisms induce the same map in cohomology. This result will be a fundamental ingredient to prove that measurable bounded cohomology is invariant under orbit equivalence and hence Theorem \ref{theorem_expo_orbit_equivalence}.

\subsection{Fiber $\mathrm{L}^1$ and $\Linf$ spaces}\label{section_bundles}
 In this section we are going to recall the definition of measurable bundle of Banach spaces. We will focus our attention mainly on the case of $\mathrm{L}^1$ and $\Linf$ bundles. For a more detailed discussion about those topics we refer the reader to either Fell and Doran \cite[Chapter II.4]{fell:doran} or Anantharaman-Delaroche and Renault \cite[Appendix A.3]{delaroche:renault}. 

Let $X$ be a standard Borel space with probability measure $\mu$. A \emph{bundle of Banach spaces} $\mathcal{E}=\{E_x\}_{x \in X}$ over $X$ is the assignment of a Banach space to every point of $X$. A \emph{section} of the bundle is a function 
$\sigma:X \rightarrow \bigsqcup_{x \in X} E_x$ such that $\sigma(x) \in E_x$ for every $x \in X$. 

\begin{defn}\label{def measurable structure}
Let $(X,\mu)$ be a standard Borel probability space and let $\mathcal{E}=\{E_x\}_{x \in X}$ be a bundle of Banach spaces over $X$. We say that a family of sections $\mathscr{M}$ is a $\mu$-\emph{measurable structure} for $\mathcal{E}$ if the following properties are satisfied:
\begin{itemize}
\item if $\sigma_1,\sigma_2 \in \mathscr{M}$, then $\sigma_1+\sigma_2 \in \mathscr{M}$,
\item if $\sigma\in \mathscr{M}$ and $\phi:X \rightarrow \mathbb{C}$ is a $\mu$-measurable function, then $\phi \cdot \sigma \in \mathscr{M}$,
\item the function $x \mapsto \lVert \sigma(x) \rVert_{E_x}$ is $\mu$-measurable, where $\lVert \ \cdot \ \rVert_{E_x}$ is the norm on $E_x$, 
\item if $\sigma_n \in \mathscr{M}$ is a sequence of sections converging $\mu$-almost everywhere to a section $\sigma$, then $\sigma \in \mathscr{M}$. 
\end{itemize}
All the above operations involving sections, such as their sum, must be intended pointwise. 
%With a slight abuse of notation, we will denote a section both as the assignment $x\mapsto \sigma(x)$ and as the family $(\sigma(x))_{x\in X}$

The pair $(\mathcal{E},\mathscr{M})$ is called \emph{measurable bundle of Banach spaces} over $X$ and the elements of $\mathscr{M}$ are called $\mu$-\emph{measurable sections}. 
We say that a measurable bundle of Banach spaces $(\mathcal{E},\mathscr{M})$ is \emph{separable} if there exists a countable family of $\mu$-measurable sections $(\sigma_n)_{n \in \mathbb{N}}$ such that the subset $\{\sigma_n(x) \ |  \ n\in \matN\}$ is dense in $E_x$ for $\mu$-almost every $x \in X$. 
\end{defn}

\begin{es} \label{esempio L1}
Let $(Y,\nu)$ and $(X,\mu)$ be two standard Borel probability spaces and consider $\pi:Y \rightarrow X$ a Borel measurable surjection such that $\pi_\ast \nu = \mu$. We fix a separable Banach space $E$. 

In virtue of Hahn disintegration theorem \cite[Theorem 2.1]{Hahn}, we know that there exists an assignment $x \mapsto \nu^x$ of measures on $Y$ such that 
$$\nu^x(Y \setminus \pi^{-1}(x))=0$$ 
and, for every Borel measurable map $f:Y \rightarrow E$, the function 
$$
x \mapsto \nu^x(\lVert f \rVert_E):= \int_Y \lVert f(y) \rVert_E d\nu^x(y) \ 
$$
is Borel measurable. Here $\lVert \ \cdot \  \rVert_E$ is the norm on $E$. 

We can consider the bundle of Banach spaces defined by
$$
\mathcal{E}:=\{ \mathrm{L}^1((Y,\nu^x),E)\}_{x \in X} \ ,
$$
where each $\mathrm{L}^1((Y,\nu^x),E)$ is the usual Bochner space of norm $\nu^x$-integrable functions. 
As noticed by Anantharaman-Delaroche and Renault \cite[Appendix A.3, Example (2)]{delaroche:renault}, the family of sections
$$
\mathscr{N}:=\{ x\mapsto \sigma(x):=f\,| \, f:Y\rightarrow E \ \textup{ Borel} \,,\, \nu^x(\lVert f \rVert_E) < \infty \ \textup{for $\mu$-a.e. $ x \in X$} \}  
$$
generates a measurable structure $\mathscr{M}$ for $\mathcal{E}$ which is separable, by the fact that $Y$ is standard and $E$ is separable. Notice that $\mathscr{M}$ is the smallest measurable structure for $\mathcal{E}$ containing the subset $\mathscr{N}$. 

We can define 
$$
\mathscr{L}^1(X,\mathcal{E}):=\left\{ \sigma\in \mathscr{M} \ | \ \int_X \nu^x(\lVert \sigma(x) \rVert_E) d\mu(x) < \infty \right\}
$$
endowed with the seminorm 
$$
\lVert \sigma \rVert_{\mathscr{L}^1(X,\mathcal{E})}:=\int_X \nu^x(\lVert \sigma(x) \rVert_E)d\mu(x) \ .
$$
If we quotient by the subset of zero normed sections, we obtain a normed space $\mathrm{L}^1(X,\mathcal{E})$ called \emph{fiber Bochner $\mathrm{L}^1$-space} over $X$ with coefficients in $\mathcal{E}$. 

If we consider $f \in \mathrm{L}^1(Y,E)$ and we define
$$
f^x(y):=\begin{cases*}
f(y)\ \ &\textup{if $ y \in \pi^{-1}(x)$}, \\
0 \ \ &\textup{otherwise}\,,
\end{cases*}
$$ for every $x\in X$, Hahn disintegration theorem \cite[Theorem 2.1]{Hahn} guarantees that 
$$ X\rightarrow \bigsqcup\limits_{x\in X} \mathrm{L}^1((Y,\nu^x),E)\,,\;\;\;  x\mapsto f^x$$ lies in $ \mathrm{L}^1(X,\mathcal{E})$.
With a slight abuse of notation we denote by $f$ also this section.
In fact, if $\lVert \ \cdot \ \rVert_{\mathrm{L}^1(X,\mathcal{E})} $
is the induced norm on $\mathrm{L}^1(X,\mathcal{E})$,
 it holds 
\begin{equation} \label{eq isometry L1}
\lVert f \rVert_{\mathrm{L}^1(Y,E)} = \int_Y \lVert f(y) \rVert_E d\nu(y) =\int_X\left( \int_Y \lVert f(y) \rVert_E d\nu^x(y) \right) d\mu(x)=\lVert f \rVert_{\mathrm{L}^1(X,\mathcal{E})} \,,
\end{equation}
and
we have an isometric isomorphism of Banach spaces
\begin{equation} \label{eq L1 isomorphism}
\Phi: \mathrm{L}^1(Y,E) \rightarrow \mathrm{L}^1(X,\mathcal{E}) \ ,
\;\;\;
f \mapsto ( x\mapsto f^x) \ .
\end{equation}
The fact that $x \mapsto f^x$ is a measurable section follows directly by Equation \eqref{eq isometry L1} and by the fact that $\mathscr{M}$ is generated by $\mathscr{N}$. 
\end{es}

Our goal is to pass to dual spaces in order to establish an isomorphism for $\Linf$-spaces similar to the one of Equation \eqref{eq L1 isomorphism}. 

\begin{defn}\label{def dual bundle}
Let $(X,\mu)$ be a standard Borel probability space and let $(\mathcal{E},\mathscr{M})$ be a measurable bundle of Banach spaces. Suppose that $\mathscr{M}$ is separable. The \emph{dual measurable bundle} $(\mathcal{E}^\ast, \mathscr{M}^\ast)$ is obtained by taking the family of dual Banach spaces $\mathcal{E}^\ast=\{E_x^\ast\}_{x \in X}$ and by considering as $\mathscr{M}^\ast$ the collection of sections $\lambda:X \rightarrow \bigsqcup E^\ast_x$ such that
$$
x \mapsto \langle \lambda(x) | \sigma(x) \rangle 
$$
is Borel measurable for every $\sigma \in \mathscr{M}$. 
\end{defn}

By \cite[Lemma A.3.7]{delaroche:renault} the separability of $\mathscr{M}$ guarantees that $\mathscr{M}^\ast$ is a well-defined measurable structure on $\mathcal{E}^\ast$. Additionally, it is sufficient to check the defining condition only on a subset $\mathscr{N}$ generating the structure $\mathscr{M}$.

\begin{es} \label{esempio Linf}
We consider two standard Borel probability spaces $(Y,\nu),(X,\mu)$  and let $\pi:Y \rightarrow X$ be a measurable surjection such that $\pi_\ast \nu = \mu$. Let $E$ be a separable Banach space and let $\{ \nu^x \}_{x \in X}$ be the disintegration of $\nu$ along $\pi$ obtained by \cite[Theorem 2.1]{Hahn}. 

Given the bundle
$$
\mathcal{E}:=\{\mathrm{L}^1((Y,\nu^x),E) \}_{x \in X}
$$
endowed with the separable measurable structure $\mathscr{M}$ of Example \ref{esempio L1}, we can construct the associated dual measurable bundle. More precisely, we define
$$
\mathcal{E}^\ast:=\{\Linfw((Y,\nu^x),E^\ast) \}_{x \in X} \ , 
$$
where each $\Linfw((Y,\nu^x),E^\ast)$ is the Bochner space of weak$^\ast$ measurable essentially bounded functions. We endow $\mathcal{E}^\ast$ with the dual measurable structure $\mathscr{M}^\ast$ given by Definition \ref{def dual bundle}.

We define 
$$
\mathscr{L}^\infty(X,\mathcal{E}^\ast):=\{ \lambda \in \mathscr{M}^\ast \ | \ x \rightarrow \lVert \lambda(x) \rVert_{\infty} \ \textup{is essentially bounded} \}, 
$$
equipped with the seminorm 
$$
\lVert \lambda \rVert_{\mathscr{L}^\infty(X,\mathcal{E}^\ast)}:=\mathrm{ess \ sup}_{x \in X} \lVert \lambda(x) \rVert_{\infty} \ .
$$
By quotienting by the subspace of zero normed sections, we obtain a normed space $\Linf(X,\mathcal{E}^\ast)$ called \emph{fiber Bochner $\mathrm{L}^\infty$-space} over $X$ with coefficients in $\mathcal{E}^\ast$. 

By \cite[Proposition A.3.9]{delaroche:renault} there exists a canonical isomorphism 
\begin{equation}\label{equation_duality}
(\mathrm{L}^1(X,\mathcal{E}))^\ast \cong \Linf(X,\mathcal{E}^\ast) \ ,
\end{equation}
where the space on the left-hand side is the dual of the fiber Bochner $\mathrm{L}^1$-space over $X$. Since it is well-known that 
$$
(\mathrm{L}^1(Y, E))^\ast \cong \Linfw(Y,E^\ast) \ ,
$$
by dualizing the isometric isomorphism of Equation \eqref{eq L1 isomorphism}, we obtain another isometric isomorphism as follows
\begin{equation} \label{eq Linf isomorphism}
\Phi^\ast: \Linfw(Y,E^\ast) \rightarrow \Linf(X,\mathcal{E}^\ast) \ ,
\;\;\;
\lambda \mapsto (x \mapsto \lambda^x) \ ,
\end{equation}
where $\lambda^x$ is defined analogously as in Example \ref{esempio L1}. 
\end{es}

\subsection{Measurable bounded cohomology.}\label{section_Measurable bounded cohomology}
Fix a measured groupoid $(\calG,\nu)$ and a measurable coefficient $\calG$-module $(E,L)$, the latter considered with the Borel structure induced by the weak$^*$ topology.
%
%\todo{Dobbiamo essere piu' precisi quando introduciamo la misura su G specificandolo quando possibile}
%
In particular we have a fixed predual $E^{\flat}$.
Consider the fiber product
$$\calG^{(\bullet)}=\calG*\ldots*\calG\coloneqq \{(g_1,\ldots,g_{\bullet})\in \calG^{\bullet} \, | \, t(g_i)=t(g_j)\;\; \forall i,j\;\}$$ endowed with the 
measurable structure inherited by $\calG^{\bullet}$. If $(\nu^x)_{x\in \calG^{(0)}}$ is the quasi-invariant $t$-disintegration of $\nu$, we denote by
$\nu^{(\bullet)}_x\coloneqq \nu^x \otimes\ldots \otimes \nu^x$ the product measure supported on $\left(\calG^{(\bullet)}\right)^x\coloneqq (t^{\bullet})^{-1}(x)$, where $t^{\bullet}$
is the fiber target map $$t^\bullet:\calG^{(\bullet)} \rightarrow \calG^{(0)}\,,\;\;\;  t^\bullet(g_1,\ldots,g_\bullet)=t(g_1)\,.$$
In this setting, the measure $\nu^{(\bullet)}$ on $\calG^{(\bullet)}$ is the convolution defined by Equation \eqref{equation_convolution_measures} between the measure $\mu$ on $\calG^{(0)}$ and the system $(\nu^{(\bullet)}_x)_{x\in \calG^{(0)}}$. 
Define the Banach space
$$\Linfw(\calG^{(\bullet)},E)\coloneqq \left\{ \lambda:\calG^{(\bullet)}\rightarrow E\, | \ \text{$\lambda$ is weak$^\ast$ measurable},\, ||\lambda||_{\infty}<+\infty \right\}\big/ \sim  $$
where two maps coinciding $\nu^{(\bullet)}$-almost everywhere are identified and $\| \lambda \|_\infty$ denotes the essential supremum on $\calG^{(\bullet)}$. 
By Example \ref{example_groupoids_action2}, $\calG$ acts in a natural way on each fiber product $\calG^{(\bullet)}$ by taking
$$
\calG \ast \calG^{(\bullet)} \rightarrow \calG^{(\bullet)} \ ,
$$
$$
(g,(g_1,\ldots,g_\bullet)) \mapsto (gg_1,\ldots,gg_\bullet) \ .
$$
Following Anantharaman-Delaroche and Renault \cite[Lemma 4.1.12]{delaroche:renault}, such an action, combined with the $\calG$-action on $E$, allows to define the notion of $\calG$-invariance. A function $\lambda \in \Linfw(\calG^{(\bullet)},E)$ is $\calG$-\emph{invariant} if 
\begin{equation}\label{eq_g_invariance_0}
L(g)\lambda(g^{-1}g_1,\ldots,g^{-1}g_\bullet)=\lambda(g_1,\ldots,g_\bullet)
\end{equation}
for $\nu^{(\bullet+1)}$-almost every $(g,g_1,\ldots,g_\bullet) \in \calG^{(\bullet+1)}$. 

Since the fiber target map is a Borel surjection satisfying $t^{(\bullet)}_\ast(\nu^{(\bullet)})=\mu$, the disintegration isomorphism of Equation \eqref{eq Linf isomorphism} applies, and we obtain
\begin{equation}\label{equation_disintegration_isomorphism}
\Linfw( \calG^{(\bullet)},E) \cong \Linf( \calG^{(0)}, \mathcal{E}) \ ,
\end{equation}
where $\mathcal{E}:=\{ \Linfw (( \calG^{(\bullet)},\nu^{(\bullet)}_x),E ) \}_{x \in \calG^{(0)}}$ is the measurable bundle of Example \ref{esempio Linf}. Recall that the explicit isomorphism of Equation \eqref{equation_disintegration_isomorphism} is obtained by considering a Borel map $\lambda:\calG^{(\bullet)} \rightarrow E$ and sending it to the section $x \mapsto \lambda^x$, where $\lambda^x$ is the extension by zero of the restriction of $\lambda$ to the fiber $(\calG^{(\bullet)})^x$. In this context, the $\calG$-invariance of Equation \eqref{eq_g_invariance_0} can be rewritten in terms of sections of the bundle $\mathcal{E}$. A section $(\lambda^x)_{x \in \calG^{(0)}}$ is $\calG$-\emph{invariant} if for almost every $g \in \calG$ it satisfies 

\begin{equation}\label{equation_g_invariance}
g\lambda^{s(g)}=\lambda^{t(g)} \ ,
\end{equation} 
where 
\begin{equation}\label{equation_action_section}
g\lambda^{s(g)}(g_1,\ldots,g_\bullet):=\begin{cases} L(g)\lambda(g^{-1}g_1,\ldots,g^{-1}g_\bullet)& \text{if } t(g_i)=t(g)\\
0 & \text{otherwise} \, .\end{cases}
\end{equation}
We denote by $\Linfw(\calG^{(\bullet)},E)^{\calG}$ the subspace of $\calG$-invariant functions, endowed with its natural structure of Banach subspace. 

The standard homogeneous coboundary operator 
$$\delta^{\bullet}:\Linfw(\calG^{(\bullet+1)},E)^\calG \rightarrow \Linfw(\calG^{(\bullet+2)},E)^\calG$$
 is the map defined as
\begin{gather*}
\delta^{\bullet} \lambda (g_1,\ldots, g_{\bullet+2} )\coloneqq \sum \limits_{i=1}^{\bullet+2} (-1)^{i-1}  \lambda (g_1,\ldots,  \widehat{g_i} ,\ldots, g_{\bullet+2} )\,.
\end{gather*}
so that the pair 
$$\left(\Linfw(\calG^{(\bullet+1)},E)^{\calG},\delta^{\bullet}\right)$$
forms a cochain complex of Banach spaces.

\begin{defn}\label{definition_measurable_bounded_cohomology}
The \emph{measurable bounded cohomology} of a measured groupoid $\calG$ with coefficients into a measurable coefficient $\calG$-module $E$ is
$$\Hmb^k(\calG,E)=\Hm^k(\Linfw(\calG^{(\bullet+1)},E)^{\calG})\, ,$$
endowed with the seminormed structure induced by the norm on $\Linfw(\calG^{(\bullet+1)},E)^{\calG}$. 
\end{defn}

\begin{oss}\label{remark_bounded_cohomology_group}
By \cite[Proposition 1.1.3]{monod:libro}, a measurable isometric action of a locally compact group $G\rightarrow \Isom(E^{\flat})$ in the sense of Definition \ref{definition_measurbale_G_module} can be promoted to a continuous action by local compactness.
In this case, $\Linfw(G^{\bullet},E)$ coincides with the space of $E$-valued weak$^*$ measurable functions on $G^{\bullet}$. Hence, by \cite[Proposition 7.5.1]{monod:libro}, we have canonical isometric isomorphisms
$$\Hmb^{\bullet}(G,E)\cong \Hcb^{\bullet}(G,E)\,,$$
where the left-hand side is the measurable cohomology of $G$, viewed as a groupoid, and the right-hand side denotes the usual continuous bounded cohomology. In this way we see that measurable bounded cohomology of groupoids boils down to the usual continuous bounded cohomology in the case of locally compact groups. 
\end{oss}

\subsection{Functoriality}\label{section_Functoriality}
In this section we show that similar homomorphisms between groupoids induces the same map in measurable bounded cohomology.

Fix $(\calG,\nu)$ and $(\calH,\tau)$ two measured groupoids, a measurable coefficient $\calG$-module $E$ and a measurable coefficient $\calH$-module $F$.
Consider a pair $(f,\phi)$, where $f:\calG\rightarrow \calH$ is a homomorphism and $\phi:\calG^{(0)} \rightarrow \calL(f^*F,E)$ is a $\calG$-map as in Definition \ref{definition_morphism_coefficient_modules}.
As observed in Remark \ref{remark_homomorphism}, we can modify $f$ on a null measure subset in order to get an algebraic homomorphism on an inessential contraction of $\calG$. Notice that a different choice of the inessential contraction of $\calG$ does not affect our reasoning, since we are going to consider class of functions coinciding almost everywhere on $\calG^{(\bullet)}$. By what we have said so far, without loss of generality, we can suppose that $f$ sends 
almost every $\calG$-unit into a $\calH$-unit. In particular, $f$ induces maps on fiber products, and this fact allows to define maps
\begin{gather}\label{equation_functoriality_of_map}
\Linfw(f,\phi)^{\bullet}:\Linfw (\calH^{(\bullet)},F)\rightarrow \Linfw (\calG^{(\bullet)},E)\\
\notag \Linfw(f,\phi)^{\bullet}(\lambda)(g_1,\ldots,g_{\bullet})\coloneqq \phi_{x}(\lambda(f(g_1),\ldots,f(g_{\bullet})))
\end{gather}
for almost every $x\in X$ and $(g_1,\ldots,g_{\bullet})\in (t^{\bullet})^{-1}(x)\subset \calG^{(\bullet)}$.

We need to check that such map is well-defined. Take $\widetilde{\lambda}$ another representative differing from $\lambda$ by a null set $Z \subset \calH^{(\bullet)}$. Since $f_*\nu$ is absolutely continuous with respect to $\tau$, its preimage through the induced map $f^{(\bullet)}:\calG^{(\bullet)} \rightarrow \calH^{(\bullet)}$ has still measure zero in $\calG^{(\bullet)}$. Such preimage is the locus where $\Linfw(f,\phi)(\lambda)$ and $\Linfw(f,\phi)(\widetilde{\lambda})$ may differ. This shows that the images of two different representatives of a class in $\Linfw (\calH^{(\bullet)},F)$ coincide in $\Linfw (\calG^{(\bullet)},E)$.
Secondly,  we need to prove that $\Linfw(f,\phi)^{\bullet}(\lambda)$ is essentially bounded, but this follows from the uniform boundedness of $\phi$.
Moreover, since the following diagram commutes
\begin{center}
\begin{tikzcd}
\Linfw (\calH^{(\bullet)},F)\arrow{rr}{\Linfw(f,\id_F)^{\bullet}} \arrow[swap]{rd}{\Linfw(f,\phi)^{\bullet}}&& \Linfw (\calG^{(\bullet)},f^*F) \arrow{ld}{\Linfw(\id_{\calG},\phi)^{\bullet}} \\
& \Linfw (\calG^{(\bullet)},E)&
\end{tikzcd}
\end{center}
 it follows that the map $\Linfw(f,\phi)^{\bullet}$ has norm at most equal to $||\phi||_{\infty}$.
We denote by $\Hmb^{\bullet}(f,\phi)$ the map induced by $\Linfw(f,\phi)^{\bullet}$ at the cohomological level. When the $\calG$-map $\phi$ will be $\id_E$ for almost every $\calG$-unit, we will sometimes drop it in the notation, and we will simply write $\Hmb^{\bullet}(f)$.

%Hence we have defined a controvariant functor
%\begin{equation}\label{equation_functor_chain}
%\Linfw(\cdot,\cdot):\textbf{GrpBB}\rightarrow  \textbf{Ch}_{\matR}
%\end{equation}
%into the category of chain complexes of Banach spaces and linear maps. 
%
%Thanks to Definition \ref{definition_measurable_bounded_cohomology} we therefore have a controvariant functor
%$$\Hmb^{\bullet}(\cdot,\cdot):\textbf{GrpBB}\rightarrow \matR\textbf{-Mod}^{\parallel\cdot \parallel}$$
%into the category of seminormed $\matR$-modules, which is nothing that the composition of the one in Equation \eqref{equation_functor_chain} with the cohomological one
%$$\Hm^{\bullet}(\cdot): \textbf{Ch}_{\matR}\rightarrow \matR\textbf{-Mod}^{\parallel\cdot \parallel}\,$$
%after taking the invariant subcomplex.

The next result shows that the measurable bounded cohomology is invariant by taking an inessential contraction.
\begin{lemma}\label{lemma_isomorphism_cohomology_inessential_contraction}
Let $\calG$ be a measured groupoid and let $E$ be a measurable coefficient $\calG$-module. If $U\subset X$ is a Borel set such that $\calG_{|U}$ is an inessential contraction of $\calG$, the inclusion map $i:\calG_{|U}\hookrightarrow \calG$ induces isometric isomorphisms
$$\mathrm{H}_{\mathrm{mb}}^{\bullet}(\calG,E)\cong \mathrm{H}_{\mathrm{mb}}^{\bullet} (\calG_{|U},E)\,.$$
\end{lemma}
\begin{proof}
If $\id_E:U\rightarrow \Isom(E)$ is the $\calG_{|U}$-map defined to be the identity on $E$ for every $x\in U$, then the induced cochain map
$$\Linfw(i,\id_E)^{\bullet}: \Linfw(\calG^{(\bullet)}, E)\rightarrow \Linfw((\calG_{|U})^{(\bullet)}, E)\,$$
gives an isometric identification between $\Linfw((\calG_{|U})^{(\bullet)}, E)$ and $\Linfw(\calG^{(\bullet)}, E)$.
\end{proof}

We are now ready to prove the main result of this section, which shows that similar homomorphisms induce the same maps in measurable bounded cohomology.

\begin{prop}\label{proposition_functoriality}
Let $\calG,\calH$ be measured groupoids and consider a measurable coefficient $\calH$-module $F$.
Let $f_0,f_1:\calG\rightarrow \calH$ be two similar homomorphisms connected by a similarity $h$ and denote by $H$ the $\calG$-map induced by $h$.
 Then 
 $$\Hmb^{\bullet}(f_1,\id_F): \Hmb^{\bullet}(\calH,F)\rightarrow \Hmb^{\bullet}(\calG,f_1^* F)$$
 and 
 $$\Hmb^{\bullet}(f_0,H):\Hmb^{\bullet}(\calH,F)\rightarrow \Hmb^{\bullet}(\calG, f_0^*F)$$
coincide. In particular
 $$\Hmb^{\bullet}(f_0,H)=\Hmb^{\bullet}(\id_{\calG},H)\circ \Hmb^{\bullet}(f_0,\id_{F})$$ and the first map on the right-hand side is an isomorphism.
\end{prop}

\begin{oss}\label{remark_funtoriality}
In the above statement we considered a unique $\calH$-module $F$ and then we endowed it with the $\calG$-actions induced by the pullback along $f_1$ and $f_0$. Here the $\calG$-maps involved are the identity for $f_1^*F$ and the identity composed with $H$ for $f_0^*F$.
It is worth noticing that it would be possible to consider a more general setting, as the map defined by Equation \eqref{equation_functoriality_of_map} allows to choose different modules and generic $\calG$-maps. For instance, one could start with two different modules $E,F$ endowed with actions by $\calG$ and $\calH$ respectively and then taking a $\calG$-map $\phi:\calG^{(0)}\rightarrow \Isom (f_1^*F,E)$. In this case the maps involved in cohomology would be
$$\Hmb^{\bullet}(f_1,\phi):\Hmb^{\bullet}(\calH,F)\rightarrow \Hmb^{\bullet}(\calG,E)$$ and 
$$\Hmb^{\bullet}(f_0,\phi\circ  H):\Hmb^{\bullet}(\calH,F)\rightarrow \Hmb^{\bullet}(\calG,E)$$
and the same argument described above shows that they coincide.

However, in order to simplify the exposition and since in the sequel we will never use such a general case, we prefer to state the result in a simpler case. 
\end{oss}

\begin{proof}
Fix an element $\lambda\in \Linfw(\calH^{(\bullet)},F)^{\calH}$, let $(g_1,\ldots,g_{\bullet})\in \calG^{(\bullet)}$ and denote by $x=t(g_i)$ and $x_i=s(g_i)$. 
Exploiting the equality $$h(t(g))f_0(g)h(s(g))^{-1}=f_1(g)$$
given by Definition \ref{definition_homomorphism_homotopy_measured_groupoid},  we have
\begin{align*}
&\;\lambda(f_1(g_1),\ldots,f_1(g_{\bullet}))\\
=&\;\lambda(h(x) f_0(g_1)h(x_1)^{-1},\ldots,h(x) f_0(g_{\bullet})h(x_{\bullet})^{-1})\\
=&\;L(h(x)) L (h(x))^{-1} \lambda(h(x) f_0(g_1)h(x_1)^{-1},\ldots,h(x) f_0(g_{\bullet})h(x_{\bullet})^{-1})\\
=&\;L(h(x)) \lambda(f_0(g_1)h(x_1)^{-1},\ldots, f_0(g_{\bullet})h(x_{\bullet})^{-1}) \ ,
\end{align*}
where the last equality follows by the $\calH$-invariance of $\lambda$. 
We now claim that $\Linfw(f_0,H)^{\bullet}$ and the map defined by the last term of the above equation induce the same map in cohomology. To prove the claim we will construct an explicit chain homotopy. To this end, for any $i=1,\ldots,\bullet$, define the maps
\begin{gather*}
s_i^{\bullet+1}: \Linfw(\calH^{(\bullet+1)},F)\rightarrow \Linfw(\calG^{(\bullet)},F)\,,\\
s_i^{\bullet+1}(\lambda)(g_1,\ldots,g_{\bullet})\coloneqq H_x (\lambda(f_0(g_1),\ldots,f_0(g_{i}),f_0(g_i)h(x_i)^{-1},\ldots,f_0(g_{\bullet})h(x_{\bullet})^{-1}))
\end{gather*}
and 
\begin{gather*}
s^{\bullet+1}: \Linfw(\calH^{(\bullet+1)},F)\rightarrow \Linfw(\calG^{(\bullet)},F)\,,\\
s^{\bullet+1}(\lambda)=\sum\limits_{i=1}^{\bullet}
(-1)^{i-1} s_i^{\bullet+1} (\lambda)\,.
\end{gather*}
Following an analogous computation as in \cite[Lemma 8.7.2]{monod:libro} we get 
\begin{align*}
s^{\bullet+1}\circ \delta_{\calH}^{\bullet-1} ( \lambda)    (g_1,\ldots,g_{\bullet})=&-\delta_{\calG}^{\bullet-2} \circ s^{\bullet}( \lambda)    (g_1,\ldots,g_{\bullet})\\
&+H_x (\lambda(f_0(g_1),\ldots,f_0(g_{\bullet})))\\
& -H_x(\lambda(f_0(g_1)h(x_1)^{-1},\ldots, f_0(g_{\bullet})h(x_{\bullet})^{-1}))
\end{align*}
and the thesis follows.

By definition, we have that 
$$ \Linfw(f_0,H)^{\bullet}(\lambda)(g_1,\ldots,g_{\bullet})=
H_{x}(\lambda(f_0(g_1),\ldots,f_0(g_{\bullet}))) \,,$$
hence we have a commutative diagram at the cohomological level
\begin{center}
\begin{tikzcd} 
\Hmb^{\bullet}(\calH,F) \arrow{rr}{\Hmb^{\bullet}(f_0,\id_{F})} \arrow[swap]{rd}{\Hmb^{\bullet}(f_0,H)}&&\Hmb^{\bullet}(\calG,f_0^* F) \arrow{ld}{\Hmb^{\bullet}(\id_{\calG},H)}\\
&\Hmb^{\bullet}(\calG, f_1^* F)& 
\end{tikzcd}
\end{center}
 where the right arrow is an isomorphism by Example \ref{example_coefficient_map_induced_by_homotopy}.
 \end{proof}
 
 As an immediate consequence we obtain the invariance of measurable bounded cohomology under similarity (and hence under isomorphism).
 \begin{cor}\label{corollary_invariance_similarity}
Let $\calG,\calH$ be measured groupoids and consider a measurable coefficient $\calH$-module $F$. If $f:\calG\rightarrow \calH$ realizes a similarity then 
$$\Hmb^\bullet(f,\id_{F}): \Hmb^{\bullet}(\calH,F)\rightarrow \Hmb^{\bullet}(\calG,f^* F)$$
are isomorphisms.
Moreover, if $f$ is an isomorphism, the above isomorphisms are isometric.
 \end{cor}
 
% \begin{oss}\label{remark_functoriality_isomorphism}
%Consider an isomorphism $f:\calG\rightarrow \calH$ between measured groupoids $\calG,\calH$ over standard probability spaces with inverse $k:\calH\rightarrow \calG$ and let $E$ be a coefficient module endowed with both a measurable $\calG$-structure and a measurable $\calH$-structure. We denote by $E_{\calG}$ and by $E_{\calH}$ the induced measurable coefficient $\calG$-module and $\calH$-module, respectively. By Corollary \ref{corollary_invariance_similarity} a $\calG$-map $\phi: \calG^{(0)}\rightarrow \Isom(f^*E_{\calH},E_{\calG})$
% defines isomorphisms 
%\begin{equation}\label{equation_example1}
%\Hmb^\bullet(f,\phi): \Hmb^{\bullet}(\calH,E_{\calH})\rightarrow \Hmb^{\bullet}(\calG,E_{\calG})\,.
% \end{equation} 
% By Example \ref{example_coefficient_module_with_different_actions}, up to modify $k$ on a null measure subset, $\phi$ induces a $\calH$-map $\psi:\calH^{(0)}\rightarrow \Isom (k^* E_{\calG},E_{\calH})$, that in turn defines isomorphisms 
%\begin{equation}\label{equation_example2}
%  \Hmb^\bullet(k,\psi): \Hmb^{\bullet}(\calG,E_{\calG})\rightarrow \Hmb^{\bullet}(\calH,E_{\calH})\,.
% \end{equation} 
% It follows immediately by the equalities
% $$\psi_{f(x)}\circ \phi_{x}=\id_E\,,\;\;\; \phi_{k(y)}\circ \psi_y =\id_E$$
% showed in Example \ref{example_coefficient_module_with_different_actions} that 
% the isomorphisms of Equation \eqref{equation_example1} and of Equation \eqref{equation_example2} are the inverse of each other.
% \end{oss}
 
We are going to describe the behavior of measurable bounded cohomology for a \emph{subgroupoid}, that is the restriction of a groupoid to a Borel subset of the unit space. A first instance of invariance is showed in Lemma \ref{lemma_isomorphism_cohomology_inessential_contraction} in the case of full-measure subgroupoids.
For $t$-discrete groupoids, taking a Borel subset of the unit space whose saturation is of full measure, we can exploit Proposition \ref{proposition_functoriality} to deduce the following 
 \begin{prop}\label{proposition_discrete}
 Let $\calG$ be a $t$-discrete groupoid over $X$. Consider a measurable coefficient $\calG$-module $E$ and $U\subset X$ a Borel subset of positive measure such that $\calG U$ is of full measure. Then the inclusion 
 $i:\calG_{|U}\hookrightarrow \calG$
 induces isomorphisms
 $$ \Hmb^{\bullet} (\calG,E) \cong \Hmb^{\bullet} (\calG_{|U},E)\,.$$
 \end{prop}
 \begin{proof}
The main goal is to show that $\calG$ and $\calG_{|U}$ are similar measured groupoids \cite[Lemma 6.6]{ramsay}. We report a sketch of Ramsay's proof for sake of completeness.

 Let $V=\calG U$.  We first claim that we can assume $V=X$. In fact, 
since $V$ has full measure, $\calG_{|V}$ is an inessential contraction of $\calG$. Hence, by Lemma \ref{lemma_isomorphism_cohomology_inessential_contraction} we have isomorphisms
$$
\Hmb^\bullet(\calG,E) \cong \Hmb^\bullet(\calG_{|V},E)\,.
$$
Assuming $\calG U=X$, it follows that $s(\calG^U)=X$, where 
$\calG^U\coloneqq \{ g\in \calG\,|\,  t(g)\in U\}$.
Up to discarding null-measure subsets and applying again Lemma \ref{lemma_isomorphism_cohomology_inessential_contraction}, by the Selection Theorem (see \cite[Corollary 18.10]{kechris} for the version needed here) there exists a Borel map $\theta:X\rightarrow\calG^U$ with 
$s(\theta(x))=x$ for every $x\in X$. Notice that the condition $s(h(x))=x$, together with the fact that $h$ has range in $\calG^U$, implies that  $h(t(g)) g h(s(g))^{-1}$ lies in $\calG_{|U}$ for every $g\in \calG$.

The discreteness of $\calG$ ensures that the map
$f:\calG\rightarrow \calG_{|U}$ defined as $f(g)\coloneqq h(t(g)) g h(s(g))^{-1}$ is actually a homomorphism (see also \cite[Lemma 2.2]{kida}). Indeed, the absolute continuity required by Definition \ref{definition_homomorphism_homotopy_measured_groupoid} follows by the fact that the target and the source of null measure subsets of discrete groupoids have null measure in the unit space.

If we consider
 the inclusion $i:\calG_{|U}\rightarrow \calG$, then $f\circ i=\id_{\calG_{|U}}$ and, by definition, $i\circ  f$ is similar to $\id_{\calG}$ via $h$. The thesis follows by Corollary \ref{corollary_invariance_similarity}.
 \end{proof}
 
 In the ergodic case the above result can be extended to every
 Borel subset of positive measure of the unit space. 
 \begin{cor}\label{corollary_isomorphism_restriction_ergodic_groupoid}
 Let $\calG$ be an ergodic $t$-discrete groupoid and $E$ a measurable coefficient $\calG$-module. Then for every Borel subset $U\subset X$ of positive measure we have isomorphisms
 $$ \Hmb^{\bullet} (\calG,E) \cong \Hmb^{\bullet} (\calG_{|U},E)\,.$$
 \end{cor}
\begin{proof}
By ergodicity, for any Borel subset $U \subset X$ of positive measure we have that $\calG U$ is of full measure and the statement follows by Proposition \ref{proposition_discrete}.
\end{proof}

In view of Example \ref{example_orbit_equivalence} we get the following
\begin{cor}\label{corollary_orbit_equivalence}
Let $\Gamma,\Lambda$ be countable groups and consider $\Gamma \curvearrowright X$ and $\Lambda\curvearrowright Y$ two essentially free measure preserving actions on standard Borel probability spaces.  Let $T:X\rightarrow Y$ be a measurable isomorphism that defines an orbit equivalence between $\Gamma \curvearrowright X$ and $\Lambda\curvearrowright Y$. Then there exist isometric isomorphisms
$$ \Hmb^{\bullet} (\Lambda\ltimes Y,E) \cong \Hmb^{\bullet} (\Gamma\ltimes X,E)$$
for every measurable coefficient $(\Lambda\ltimes Y)$-module $E$.

If both actions are also ergodic, then a weak orbit equivalence induces isomorphisms
$$ \Hmb^{\bullet} (\Lambda\ltimes Y,E) \cong \Hmb^{\bullet} (\Gamma\ltimes X,E)\,.$$
\end{cor}

\begin{proof}
Recall that by Example \ref{example_orbit_equivalence} the semidirect groupoids $\Lambda\ltimes Y$ and $\Gamma\ltimes X$ have isomorphic inessential contractions. 
As a consequence the $\Linfw$-complexes are isometric and so are their cohomology.
The second part follows by the first one and by Corollary \ref{corollary_isomorphism_restriction_ergodic_groupoid}.
\end{proof}

%\begin{oss}
%As a consequence of Remark \ref{remark_funtoriality}, in the statement of Corollary \ref{corollary_orbit_equivalence} one could consider the more general case of a coefficient module $E$ endowed with two actions by $\calG\coloneqq \Gamma\ltimes X$ and by $\calH\coloneqq \Lambda\ltimes Y$.
%In addiction, instead of the map $x\mapsto \phi_x$ with $\phi_x\equiv \id_E$, one could take a $\calG$-map $$\phi: X\rightarrow \Isom (f^*E_{\calH},E_{\calG}) \ , $$ or equivalently, as observed in Example \ref{example_coefficient_module_with_different_actions}, a $\calH$-map $$\psi:Y\rightarrow \Isom(k^*E_{\calG},E_{\calH}) \ , $$ where $f:\Gamma\ltimes X\rightarrow \Lambda\ltimes Y$ is the groupoid isomorphism of Example \ref{example_orbit_equivalence} and $k$ denotes its inverse.
%In this case, as observed in Remark \ref{remark_functoriality_isomorphism}, one would have isomorphisms
%$$\Hmb^{\bullet}(f,\phi): \Hmb^{\bullet} (\Gamma\ltimes X,E) \cong \Hmb^{\bullet} (\Lambda\ltimes Y,E) $$
%and 
%$$\Hmb^{\bullet}(k,\psi): \Hmb^{\bullet} (\Lambda\ltimes Y,E) \cong \Hmb^{\bullet} (\Gamma\ltimes X,E)$$
%of norm at most $||\phi||_{\infty}$.
%\end{oss}

%\todo{A: vedere se con questo esempio non confondiamo di piu' il lettore}

As a consequence of Corollary \ref{corollary_orbit_equivalence}, we also can compare the measurable bounded cohomology of two \emph{measure equivalent groups}. 
Given two countable groups $\Gamma,\Lambda$, a \emph{coupling} between them is a standard Borel (infinite) measure space $(\Omega,m)$ with a measure preserving $(\Gamma\times \Lambda)$-action such that both the actions of $\Gamma$ and $\Lambda$ admit finite measure fundamental domains $X_{\Gamma}$ and $X_{\Lambda}$, respectively. 
In this case we say that $\Gamma$ and $\Lambda$ are \emph{measure equivalent}.

\begin{cor}\label{corollary_measure_equivalence}
Let $\Gamma,\Lambda$ be measure equivalent countable groups. Then there exist standard Borel probability spaces $X$ and $Y$ and essentially free ergodic actions $\Gamma\curvearrowright X$, $\Lambda\curvearrowright Y$ inducing isometric isomorphisms
$$ \Hmb^{\bullet} (\Lambda\ltimes Y,E) \cong \Hmb^{\bullet} (\Gamma\ltimes X,E)$$
for every measurable coefficient $(\Lambda\ltimes Y)$-module $E$.
\end{cor}
\begin{proof}
By \cite[Theorem 2.5]{furman}, there exist standard probability spaces $X$ and $Y$ and essentially free ergodic actions $\Gamma\curvearrowright X$, $\Lambda\curvearrowright Y$ which are weakly orbit equivalent. Hence the thesis follows by the second part of Corollary \ref{corollary_orbit_equivalence}. 
\end{proof}

\begin{oss}\label{remark_ergodic_relation_weak_invariant}
By putting together Proposition \ref{proposition_functoriality} and Proposition \ref{proposition_discrete}, one has that weakly isomorphic $\mathrm{II}_1$-equivalence relations have the same measurable bounded cohomology, where two equivalence relations are \emph{weakly isomorphic} if the corresponding groupoids are weakly isomorphic in the sense of Definition \ref{definition_homomorphism_homotopy_measured_groupoid}.
Hence, the measurable bounded cohomology is an invariant for weakly isomorphic $\mathrm{II}_1$-relations. 
\end{oss}

\section{The exponential law}\label{section_The_exponential_law}
The purpose of this section is to prove Theorem \ref{theorem_expo}, where we  compare the measurable bounded cohomology of a semidirect groupoid with the continuous bounded cohomology of the group with twisted coefficients.

Consider a locally compact second countable group $G$ and a standard Borel $G$-space $X$ with a quasi invariant measure $\mu$. 
Denote by $\calG\coloneqq G\ltimes X$ the semidirect groupoid associated to the action and let $E$ be a measurable coefficient $\calG$-module with a $\calG$-action $L$. Fix a predual $E^{\flat}$ with action $L^{\flat}$ so that 
$L$ is the contragradient to $L^{\flat}$.
Consider now the Banach space $\Linfw(X,E)$ of essentially bounded $E$-valued measurable functions on $X$ (identified up to null sets) endowed with the $G$-action 
\begin{equation}\label{equation_coefficient_G_module}
G\rightarrow \Isom (\Linfw(X,E))\,,\;\;\; g\cdot\lambda(x)\coloneqq L(g,g^{-1}x)\lambda (g^{-1} x)\,.
\end{equation}
Since the $G$-action on $\Lone(X,E^{\flat})$ defined as
$$g\cdot\lambda(x)\coloneqq L^{\flat}(g,g^{-1}x)\lambda (g^{-1} x) \frac{d g^{-1}\mu}{d \mu }(x)$$
is measurable and $G$ is locally compact, it is also continuous by \cite[Proposition 1.1.3]{monod:libro}. 
It follows that the dual space $\Linfw(X,E)=\Lone(X,E^{\flat})^*$ has a natural structure of coefficient $G$-module in the sense of Monod.

We are now ready to prove the main result of this section.

\begin{rec_thm}[\ref{theorem_expo}]
Let $G$ be a locally compact second countable group, let $X$ be a Lebesgue $G$-space and denote by $\calG=G\ltimes X$ the associated measured semidirect groupoid.
If $E$ is a measurable coefficient $\calG$-module we have canonical isometric isomorphisms
\begin{equation}
\Hmb^{\bullet}(\calG,E)\cong \Hcb^{\bullet}(G,\Linfw(X,E))
\end{equation}
where the right-hand side denotes the continuous bounded cohomology of $G$ with coefficient in $\Linfw(X,E)$.
\end{rec_thm}
\begin{proof}
We recall that the measurable bounded cohomology of 
$\calG$ is computed by the complex 
\begin{equation}\label{equation_complex_1}
\left(\Linfw(\calG^{(\bullet+1)},E)^{\calG},\delta^{\bullet}\right)
\end{equation}
where 
\begin{align}\label{measured_isomorphism_product}
\calG^{(\bullet)}&= \{ ((g_1,x_1),\ldots,(g_{\bullet},x_{\bullet})) \in \calG^{\bullet} \, | \, t(g_i,x_i)=t(g_j,x_j) \; \forall i,j\}\\
&= \{ ((g_1,x_1),\ldots,(g_{\bullet},x_{\bullet})) \in \calG^{\bullet} \, | \, g_ix_i=g_jx_j \; \forall i,j\} \nonumber \\
&=\{ ((g_1,g_1^{-1}x),\ldots,(g_{\bullet},g_{\bullet}^{-1} x)) \, | \, g_i\in G \,,\, x\in X\}\cong G^{\bullet}\times X \nonumber \,,
\end{align}
and the isomorphism holds as measure spaces. 
 
By definition, an invariant function in $\Linfw(\calG^{(\bullet)} ,E)^{\calG}$ satisfies
\begin{gather*}
\notag L(g,x)\lambda((g^{-1}g_1,g_1^{-1}gx),\ldots,(g^{-1}g_{\bullet},g_\bullet^{-1}gx))= \lambda ((g_1,g_1^{-1}gx),\ldots,(g_{\bullet},g_\bullet^{-1}gx))
\end{gather*}
for almost every $g,g_1,\ldots,g_{\bullet}$ in $G$ and almost every $x\in X$.

On the other hand, if we endow $\Linfw(X,E)$ with the structure of coefficient $G$-module given by Equation \eqref{equation_coefficient_G_module}, by \cite[Proposition 7.5.1]{monod:libro} the continuous bounded cohomology of $G$ with coefficients in $\Linfw(X,E)$ is computed via the $G$-invariant subcomplex
\begin{equation}\label{equation_complex_2}
\left(\Linfw(G^{\bullet+1},\Linfw(X,E))^{G},\delta^{\bullet}\right)\,,
\end{equation}
where the $G$-action is given by 
\begin{align*}
g \cdot \lambda (g_1,\ldots,g_{\bullet+1})(x)&\coloneqq
L(g,g^{-1}x)\lambda (g^{-1}g_1,\ldots,g^{-1}g_{\bullet+1})(g^{-1}x)\,.
\end{align*}

We claim that the complexes of Equation \eqref{equation_complex_1} and Equation \eqref{equation_complex_2} are isomorphic and the cochain complex isomorphisms are realized by the following cochain maps:
\begin{gather*}
\alpha^{\bullet}: \Linfw(\calG^{(\bullet)},E)^{\calG}\rightarrow \Linfw(G^{\bullet},\Linfw(X,E))^{G}\,,\\ \alpha^{\bullet}(\lambda)(g_1,\ldots,g_{\bullet})(x)\coloneqq \lambda((g_1,g_1^{-1}x),\ldots,(g_\bullet,g_\bullet^{-1}x))
\end{gather*}
and 
\begin{gather*}
\beta^{\bullet}:\Linfw(G^{\bullet},\Linfw(X,E))^{G} \rightarrow \Linfw(\calG^{(\bullet)},E)^{\calG}\,,\\  \beta ^{\bullet}(\lambda)((g_1,g_1^{-1}x),\ldots,(g_\bullet,g_\bullet^{-1}x))\coloneqq \lambda(g_1,\ldots,g_{\bullet})(x)\,.
\end{gather*}
%In fact, as observed by Monod in \cite[Corollary 2.3.3]{monod:libro}, a consequence of Dunford--Pettis theorem (see for instance \cite[Chapter VI.8]{dunford:schwartz}) is that both the spaces in Equation \eqref{equation_expo} are canonical dual of the space of linear functionals $$\calL(\Lone(\mu\otimes \rho^{\bullet}),E)\,,$$
%where $\mu$ is a measure on $X$ as in Definition \ref{definition_regular_space} and $\rho^{\bullet}$ is the product of the Haar 
%measure on $G$.
%Hence, measurability and essentially boundedness of a (class of) $E$-valued function on $\calG^{(\bullet)}=G^{\bullet}\times X$ corresponds to the measurability and essentially boundedness of the corresponding (class of) $\Linfw(X,E)$-valued function on $\calG^{\bullet}$.
We start checking that those maps are well-defined. Notice first that the function spaces 
$$
\Linfw(\calG^\bullet,E) \cong \Linfw(G^\bullet,\Linfw(X,E))
$$
are isomorphic via $\alpha$ and $\beta$ thanks to Equation \eqref{measured_isomorphism_product} and \cite[Corollary 2.3.3]{monod:libro}. We have to verify that those maps preserve invariants. 
To see that $\alpha^{\bullet}\lambda$ is $G$-invariant, take $g\in G$. We have that  
\begin{align*}
g \cdot (\alpha^{\bullet} \lambda)(g_1,\ldots,g_{\bullet})(x)&= L(g,g^{-1}x)\alpha^{\bullet}\lambda(g^{-1}g_1,\ldots,g^{-1}g_{\bullet})(g^{-1}x)  \\
&= L(g,g^{-1}x)\lambda((g^{-1}g_1,g_1^{-1}x),\ldots,(g^{-1}g_{\bullet},g^{-1}_\bullet x))\\
&= L(g,y)\lambda((g^{-1}g_1,g_1^{-1}gy),\ldots,(g^{-1}g_{\bullet},g^{-1}_\bullet gy))\\	
&=\lambda((g_1,g_1^{-1}x),\ldots,(g_{\bullet},g^{-1}_\bullet x))\\
&=\alpha^{\bullet}\lambda (g_1,\ldots,g_{\bullet})(x)\, ,
\end{align*}
where we moved from the first line to the second one using the definition of $\alpha$, from the second line to the third one setting $x=gy$ and we concluded using the $\calG$-invariance of $\lambda$ and the definition of $\alpha$.

Conversely, if $(g,x)\in \calG$ we have 
\begin{align*}
&L(g,x)\beta^{\bullet}\lambda((g^{-1}g_1,g_1^{-1}gx),\ldots,(g^{-1}g_{\bullet},g_\bullet^{-1}gx)) \\
=&\; L(g,x)\lambda(g^{-1}g_1,\ldots,g^{-1}g_{\bullet})(x) \\
=&\;L(g,g^{-1}y)\lambda(g^{-1}g_1,\ldots,g^{-1}g_\bullet)(g^{-1}y)\\
 =&\; \lambda(g_1,\ldots,g_{\bullet})(y)\\
=&\; \beta^{\bullet} \lambda((g_1,g_1^{-1}gx),\ldots,(g_{\bullet},g_\bullet^{-1}gx)) \ ,
\end{align*}
where we moved from the first line to the second one using the definition of $\beta$, then we changed the variable setting $y=gx$ and we concluded exploiting the $G$-invariance of $\lambda$ and the definition of $\beta$. This shows that $\alpha^{\bullet}$ and $\beta^{\bullet}$ are well-defined.

It follows directly from the definition that, if $\delta^{\bullet}_{\calG},\delta^{\bullet}_{G}$ denote the coboundary operators, we have
$$\delta^{\bullet}_{\calG}\circ \alpha^{\bullet+1}=\alpha^{\bullet+2}\circ \delta^{\bullet}_{G}$$
and
$$\delta^{\bullet}_{G}\circ \beta^{\bullet+1}=\beta^{\bullet+2}\circ \delta^{\bullet}_{\calG}\,,$$
hence $\alpha^{\bullet}$ and $\beta^{\bullet}$ are also cochain maps.
Finally, it is evident that $$\alpha^{\bullet}\circ\beta^{\bullet}=	\id_{\Linfw(G^{\bullet},\Linfw(X,E))^{G}}$$ and $$\beta^{\bullet}\circ\alpha^{\bullet}=	\id_{\Linfw(\calG^{(\bullet)},E)^{\calG}}\,,$$ so the induced maps in cohomology are isomorphisms.

By Equation \eqref{measured_isomorphism_product} the cochain maps $\alpha^{\bullet}$ and $\beta^{\bullet}$ are isometric and they induce isometries at the cohomological level. This concludes the proof.
\end{proof}

%\begin{oss}\label{remark_expo}
%We notice that the maps $\alpha^{\bullet}$ and $\beta^{\bullet}$ which realize the isomorphisms between the cochain complexes in the proof of Theorem \ref{theorem_expo} are defined only at the level of the invariant subcomplexes (with respect to the $\calG$ and to the $G$-actions). Even if there is no obstruction to extend them the whole complexes, to prove the equivariance one should first choose a suitable groupoid map $G\rightarrow G\ltimes X$. However, the absence of a canonical choice for such map does not allow to follow this argument in the proof of Theorem \ref{theorem_expo}.
%\end{oss}

Imitating what happens for groups, we want to relate our measurable bounded cohomology to Westman's measurable cohomology \cite{westman69} by defining a \emph{comparison map}.
Precisely, given a measured groupoid $\calG$ and a separable Banach module $E$ with trivial $\calG$-action (namely $g\cdot v=v$ for almost every $g\in \calG$ and every $v\in E$), the
Westman's measurable cohomology is denoted by $\mathrm{H}_{\textup{m}}^\bullet(\calG,E)$ and defined as the cohomology of the complex
$$(\textup{L}^0(\overline{\calG}^{\bullet},E),\overline{\delta}^{\bullet}) \, .$$
Here $\overline{\calG}^\bullet$ is the space of \emph{paths} in $\calG$, namely of tuples $(g_1,\ldots,g_\bullet)$ where $t(g_{i+1})=s(g_i)$ for $i=1,\ldots,\bullet-1$. We set $\overline{\calG}^0\coloneqq X$. Then $\textup{L}^0(\overline{\calG}^\bullet,E)$ is the space of Borel functions identified when they coincide up to null sets and $\overline{\delta}^{\bullet}$ is the inhomogeneous coboundary operator. Following a similar computation as in Monod \cite[Chapter 7]{monod:libro} or as in Blank \cite[Proposition 3.2.23]{blank}, 
one can show that
the cochain map
\begin{gather*}
\vartheta: \textup{L}^0(\overline{\calG}^{\bullet},E)\rightarrow \textup{L}^0(\calG^{(\bullet+1)},E)^{\calG} \,, \\
 \vartheta\lambda (g_0,\ldots,g_{\bullet})\coloneqq\lambda(g_0^{-1}g_1,g_1^{-1}g_2,\ldots,g_{\bullet-1}^{-1}g_{\bullet})
\end{gather*}
with 
$
 \vartheta\lambda (g)\coloneqq L(t(g))\lambda(t(g))
$ for every $\lambda\in  \textup{L}^0(\overline{\calG}^0,E)$
%and
%\begin{gather*}
%\vartheta: \textup{L}^0(\calG^{(\bullet+1)},E)^{\calG}  \rightarrow \textup{L}^0(\overline{\calG}^{\bullet},E)\,, \\
% \vartheta \lambda (g_1,\ldots,g_{\bullet})\coloneqq \lambda(t(g_1) ,g_1,g_1 g_2,\ldots,g_1\cdots g_{\bullet})\,
%\end{gather*}
realizes an isomorphism between $\mathrm{H}_{\textup{m}}^\bullet(\calG,E)$ and the cohomology of the complex of $\calG$-invariants
$$(\textup{L}^0(\calG^{(\bullet+1)},E)^{\calG},\delta^{\bullet})\,,$$
where $\delta^{\bullet}$ is the homogeneous coboundary operator. 
%Given $\lambda\in \textup{L}^0(\calG^{(\bullet+1)},E)^{\calG}$, a priori
%the value of $\vartheta \lambda$ on $(g_1,\ldots,g_{\bullet})\in \overline{\calG}^{\bullet}$ depends only on $(t(g_1),g_1,g_1g_2,\ldots,g_1\cdots g_{\bullet})$. The latter lies in the set $\{(x,g_1,\ldots,g_{\bullet})\in X\times \calG^{(\bullet)}\,|\, t(g_i)=x\}$, which may have null measure (for instance when $\calG$ is a group). Nevertheless, we can circumvent this problem exploiting the $\calG$-invariance of $\lambda$, that allows to vary the first entry of the tuple $(t(g_1),g_1,g_1g_2,\ldots,g_1\cdots g_{\bullet})$ and ensure us the well-posedness of $\vartheta \lambda$.

In this setting, the inclusion $\Linf(\calG^{(\bullet)},E)\hookrightarrow \textup{L}^0(\calG^{(\bullet)},E)$ (where we dropped the weak$^*$ measurability in the domain by the separability of $E$ \cite[Lemma 3.3.3]{monod:libro}) is a cochain map and it preserves $\calG$-invariance. As a consequence it induces maps at a cohomological level.

\begin{defn}\label{definition comparison map}
Let $\calG$ be a measured groupoid and let $E$ be a separable Banach module with trivial $\calG$-action. The \emph{comparison map} 
$$
\mathrm{comp}^\bullet_\calG: \Hmb^{\bullet}(\calG,E)\rightarrow \Hm_{\textup{m}}^{\bullet}(\calG,E)
$$
is the map induced in cohomology by the inclusion of complexes $\Linf(\calG^{(\bullet+1)},E)^\calG \hookrightarrow \mathrm{L}^0(\calG^{(\bullet+1)},E)^\calG$.
\end{defn}

\begin{es}\label{es_comparison_groups}
Let $\calG=G$ be a locally compact group. We already observed (Remark \ref{remark_bounded_cohomology_group}) that for a separable Banach trivial $G$-module, the measurable bounded cohomology $\Hmb^\bullet(G,E)$ coincides with its continuous variant $\Hcb^\bullet(G,E)$. By the work by Austin and Moore \cite{AM13}, the cohomology of the complex $(\mathrm{L}^0(G^{\bullet+1},E)^G,\delta^\bullet)$ is isomorphic to the continuous cohomology of $G$, namely $\mathrm{H}_{\textup{c}}^\bullet(G,E)$. As a consequence, in this context, the comparison map $\mathrm{comp}_\calG^\bullet$ boils down to the usual comparison map defined in the continuous setting, namely $\mathrm{comp}^\bullet_G:\Hcb^\bullet(G,E) \rightarrow \mathrm{H}^\bullet_{\textup{c}}(G,E)$.
\end{es}

Consider now the case of a semidirect groupoid $\Gamma\ltimes X$, where $\Gamma$ is a discrete countable group and $X$ is a standard Borel $\Gamma$-space.
By \cite[Theorem 1.0]{westman:71}, one has an exponential law for the measurable cohomology of $\Gamma \ltimes X$, namely 
$$\Hm_{\textup{m}}^{\bullet}(\Gamma\ltimes X,E)\cong \Hm^{\bullet}(\Gamma,\textup{L}^0(X,E))\,,$$
where the right-hand side denotes the cohomology of the group $\Gamma$ with coefficients in the space of $E$-valued measurable functions $\textup{L}^0(X,E)$. The latter is endowed with the usual left regular action $g\cdot \lambda(x)=\lambda(g^{-1} x)$.

\begin{prop} \label{proposition comparison map}
Let $\Gamma$ be a discrete countable group and let $(X,\mu)$ be a standard Borel $\Gamma$-space. The comparison map associated to the semidirect groupoid $\calG:=\Gamma \ltimes X$ is conjugated to the composition of the comparison map $\Hb^\bullet(\Gamma,\mathrm{L}^\infty(X,E)) \rightarrow \mathrm{H}^\bullet(\Gamma,\mathrm{L}^\infty(X,E))$ with the map induced by the change of coefficients $\mathrm{L}^\infty(X,E) \hookrightarrow \mathrm{L}^0(X,E)$.
\end{prop}
\begin{proof}
Thanks to our Theorem \ref{theorem_expo} and Westman's exponential law \cite[Theorem 1.0]{westman:71}, we have the following commutative diagram at the level of cochains
$$
\xymatrix{
\mathrm{L}^\infty(\calG^{(\bullet+1)},E)^\calG \ar[rr]^\cong \ar[dd] && \ell^\infty(\Gamma^{\bullet+1},\mathrm{L}^\infty(X,E))^\Gamma \ar[d]\\
&& \ell^0(\Gamma^{\bullet+1},\mathrm{L}^\infty(X,E))^\Gamma \ar[d] \\
\mathrm{L}^0(\calG^{(\bullet+1)},E)^\calG \ar[rr]^\cong && \ell^0(\Gamma^{\bullet+1},\mathrm{L}^0(X,E))^\Gamma \ ,
}
$$
where $\ell^\infty$ and $\ell^0$ denote the space of bounded and unbounded functions, respectively. The left arrow is the inclusion, the top isomorphism is Theorem \ref{theorem_expo}, the bottom isomorphism is Westman exponential law. Finally the composition on the right is given by the inclusion followed by the change of coefficients. 

Passing at the level of cohomology groups, we obtain
\begin{center}
\begin{tikzcd}
\Hmb^{\bullet}(\calG,E) \arrow{dd}[swap]{\textup{comp}_{\calG}^{\bullet}}   \arrow{rr}{\cong} &&   \Hb^{\bullet}(\Gamma,\Linf(X,E))  \arrow{d}{\textup{comp}_\Gamma^{\bullet} } \\
&&  \Hm^{\bullet}(\Gamma,\Linf(X,E))  \arrow{d} \\
\Hm_{\textup{m}}^{\bullet}(\calG,E)  \arrow{rr}{\cong}  && \Hm^{\bullet}(\Gamma,\textup{L}^0(X,E))\,.
\end{tikzcd}
\end{center}
Here $\textup{comp}_\Gamma^{\bullet}$ is the usual comparison map, whereas the right-down map is the function induced by the inclusion of coefficients $\Linf(X,E)\hookrightarrow \textup{L}^0(X,E)$. This concludes the proof. 
\end{proof}

\section{Main consequences of the exponential law}\label{section_Main consequences of the exponential law}

In this section we are going to list several consequences of Theorem \ref{theorem_expo}. We start relating the bounded cohomology of orbit equivalent actions proving Theorem \ref{theorem_expo_orbit_equivalence}, which follows directly by Proposition \ref{proposition_functoriality} and by Theorem \ref{theorem_expo}. This result has interesting applications also in case of measure equivalence, since it allows to compare the bounded cohomology of measure equivalent groups (see Corollary \ref{corollary_measure_equivalence}). If we consider different actions of the same group, we prove an injectivity result for factors in degree $2$ (Proposition \ref{prop_factor}). Then, in Proposition \ref{proposition_homogeneous_semidirect_product}, we prove an isomorphism between the measurable bounded cohomology of the group action on a homogeneous space and the bounded cohomology of the stabilizer of a point. For groupoids arising from higher rank simple lie groups, we obtain a characterization of their cohomology in terms of lattices (Corollary \ref{cor_stuck_zimmer}). Finally, we exhibit vanishing results for the measurable bounded cohomology of several semidirect groupoids (see Proposition \ref{proposition_amenable_space}, Proposition \ref{proposition_ergodic_amenable_relation}, Proposition \ref{proposition_irreducible_higher_rank} and Proposition \ref{proposition_lattice_products}).

We start with the proof of

\begin{rec_thm}[\ref{theorem_expo_orbit_equivalence}]
Let $\Gamma,\Lambda$ be two countable groups and consider $\Gamma \curvearrowright X$ and $\Lambda \curvearrowright Y$ two essentially free measure preserving actions on standard Borel probability spaces. Denote by $\calG\coloneqq \Gamma\ltimes X$ and by $\calH\coloneqq \Lambda\ltimes Y$. 
If $T:X\rightarrow Y$ is a measurable isomorphism defining an orbit equivalence between $\Gamma \curvearrowright X$ and $\Lambda \curvearrowright Y$, then for every measurable coefficient $\calH$-module $E$ we have isometric isomorphisms
$$
\Hb^{\bullet}(\Lambda,\Linfw(Y,E)) \cong \Hb^{\bullet}(\Gamma,\Linfw(X,E)) \ ,
$$
induced by the cochain maps
\begin{gather*}
 \mathcal{J}^{\bullet+1}:\Linfw(\Lambda^{\bullet+1}, \Linfw(Y,E))^{\Lambda} \rightarrow  \Linfw(\Gamma^{\bullet+1}, \Linfw(X,E))^{\Gamma}\,\\
 \mathcal{J}^{\bullet+1}\theta (\gamma_1,\ldots,\gamma_{\bullet})(x)= 
\theta (\sigma_T(\gamma_1,\gamma_1^{-1}x),\ldots,\sigma_T(\gamma_{\bullet+1},\gamma_{\bullet+1}^{-1}x))(T(x))
\end{gather*}
for every $\gamma_1,\ldots,\gamma_{\bullet+1}$ in $\Gamma$ and $x\in X$ and $\sigma_T:\Gamma\times X\rightarrow \Lambda$ is the cocycle associated to the orbit equivalence.

If the actions are also ergodic but only weakly orbit equivalent, then we have similar isomorphisms
$$
\Hb^{\bullet}(\Lambda,\Linfw(Y,E)) \cong \Hb^{\bullet}(\Gamma,\Linfw(X,E)) \ .
$$
\end{rec_thm}

\begin{proof}
By Corollary \ref{corollary_orbit_equivalence} an orbit equivalence between $\Gamma \curvearrowright X$ and $\Lambda \curvearrowright Y$ determines an isomorphism of the measurable bounded cohomology of the semidirect products, namely
$$
\Hmb^\bullet(\Lambda \ltimes Y,E) \cong \Hmb^\bullet(\Gamma \ltimes X,E) \ . 
$$
Applying Theorem \ref{theorem_expo} twice, once on each side, we get the desired isomorphism. The formula given at the level of cochains follows by an easy diagram chasing. An explicit computation leads to 
\begin{align*}
 \mathcal{J}^{\bullet+1}\theta (\gamma_1,\ldots,\gamma_{\bullet+1})(x)\coloneqq &\; (\alpha^{\bullet+1}_{\calG} \circ \Linfw(f_T, \id_E)^{\bullet+1} \circ \beta^{\bullet+1}_{\calH}) \theta(\gamma_1,\ldots,\gamma_{\bullet+1})(x)\\
=  &\;(\alpha^{\bullet+1}_{\calG} \circ \Linfw(f_T, \id_E)^{\bullet+1}) \theta ((\gamma_1,\gamma_1^{-1}x),\ldots,(\gamma_{\bullet+1},\gamma_{\bullet+1}^{-1}x))\\
= &\;\alpha^{\bullet+1}_{\calG} \theta (f_T(\gamma_1,\gamma_1^{-1}x),\ldots,f_T(\gamma_{\bullet+1},\gamma_{\bullet+1}^{-1}x))\\
= &\;\alpha^{\bullet+1}_{\calG} \theta ((\sigma_T(\gamma_1,\gamma_1^{-1}x), T(\gamma_1^{-1}x)),\ldots\\
&\ldots,(\sigma_T(\gamma_{\bullet+1},\gamma_{\bullet+1}^{-1}x), T(\gamma_{\bullet+1}^{-1}x)))\\
= &\;\theta (\sigma_T(\gamma_1,\gamma_1^{-1}x),\ldots,\sigma_T(\gamma_{\bullet+1},\gamma_{\bullet+1}^{-1}x))(T(x))
\end{align*}
where to pass from the third line to the fourth one we applied the definition of $f_T$, and we moved from the fourth line to the fifth one by exploiting the fact $\sigma_T(\gamma,\gamma^{-1}x)T(\gamma^{-1}x)=T(x)$. 

The proof in the case of a weak orbit equivalence is analogous, but one needs to consider the similarity of Proposition \ref{proposition_discrete}.
\end{proof}

Exploiting the previous theorem, we can prove that measure equivalent countable groups admit actions whose measurable bounded cohomology coincide. 
\begin{cor}\label{corollary_measure_equivalence_induction}
Let $\Gamma$ and $\Lambda$ be measure equivalent countable groups. Then there exist two standard Borel probability spaces $X$ and $Y$ and essentially free ergodic probability measure preserving actions $\Gamma \curvearrowright X$ and $\Lambda \curvearrowright Y$ inducing isomorphisms
$$
\Hb^{\bullet}(\Lambda,\Linfw(Y,E)) \cong \Hb^{\bullet}(\Gamma,\Linfw(X,E)) \ ,
$$
for every measurable coefficient $(\Lambda \ltimes Y)$-module $E$.  
\end{cor}

\begin{proof}
By \cite[Theorem 2.5]{furman} we know that there exist standard Borel probability spaces $X$ and $Y$ and essentially free probability measure preserving ergodic actions $\Gamma \curvearrowright X$ and $\Lambda \curvearrowright Y$ which are weakly orbit equivalent. We can apply Theorem \ref{theorem_expo_orbit_equivalence} to conclude. 
\end{proof}

\begin{oss}\label{remark_monod_shalom}
An isomorphism in bounded cohomology in case of measure equivalence has already been proved by Monod and Shalom \cite[Theorem 4.6]{MonShal}, and in some sense it is related to Corollary \ref{corollary_measure_equivalence}. Precisely, given a coupling $\Omega$ between $\Gamma$ and $\Lambda$ and a coefficient $(\Gamma\times \Lambda)$-module $E$ (see Remark \ref{remark_bounded_cohomology_group}), they provided isometric isomorphisms
\begin{equation}\label{equation_isom_mon_shal}
\Hb^{\bullet}(\Lambda,\Linfw(\Omega,E)^{\Gamma}) \cong\Hb^{\bullet}(\Gamma,\Linfw(\Omega,E)^{\Lambda}) \,.
\end{equation}
In the case when the $(\Gamma\times \Lambda)$-action on $E$ is trivial, we have 
that $\Linfw(\Omega,E)^{\Gamma}\cong \Linfw(\Gamma\backslash \Omega,E)$ and $\Linfw(\Omega,E)^{\Lambda}\cong\Linfw(\Lambda\backslash \Omega,E)$ and Equation \eqref{equation_isom_mon_shal} becomes
\begin{equation}\label{equation_isom_mon_shal2}
\Hb^{\bullet}(\Lambda,\Linfw(\Gamma\backslash \Omega,E))\cong\Hb^{\bullet}(\Gamma,\Linfw(\Lambda\backslash \Omega,E)) \,.
\end{equation}

In this setting, the proof of \cite[Theorem 2.5]{furman} 
shows that one can take an extension $\Omega'$ of $\Omega$ such that $\Lambda\curvearrowright\Gamma\backslash\Omega'$ and $\Gamma\curvearrowright\Lambda\backslash \Omega'$ are essentially free. 
Since $E$ inherits a natural structure of measurable coefficient $(\Lambda\ltimes \Gamma\backslash\Omega')$-module where no twist appears in the action, we can apply Corollary \ref{corollary_measure_equivalence} getting an isomorphism
\begin{equation}\label{equation_isom_mon_shal3}
\Hb^{\bullet}(\Lambda,\Linfw(\Gamma\backslash \Omega',E))\cong\Hb^{\bullet}(\Gamma,\Linfw(\Lambda\backslash \Omega',E)) \,.
\end{equation}

To conclude, in our result we consider a wider family of coefficient modules, since we admit parametrized actions depending on the standard Borel space, but we only provided isomorphisms for an extended coupling for which the actions are essentially free. When the action does not depend on the parameter, that is when it comes from a representation, our result follows by \cite[Proposition 4.6]{MonShal}, which holds for any coupling but only for non-twisted actions.
Moreover, in general the involved actions are only weakly isomorphic, but a priori nothing can be said about the norms.
\end{oss}

Something relevant can be said also for factors (Example \ref{example_factor}), but only when the cohomological degree is equal to $2$. 

\begin{prop}\label{prop_factor}
Let $\Gamma$ be a countable group and consider two measure preserving actions $\Gamma \curvearrowright (X,\mu_X),\Gamma \curvearrowright (Y,\mu_Y)$ on two standard Borel probability spaces. Define $\calG:=\Gamma \ltimes X$ and $\calH:=\Gamma \ltimes Y$. If $Y$ is a factor of $X$, then for any separable Hilbert coefficient $\calH$-module $E$ we have that 
$$
\Hmb^2(f_T):\Hmb^2(\calH,E) \rightarrow \Hmb^2(\calG,E) \ ,
$$
is injective, where $f_T:\calG\rightarrow \calH$ is the map induced by the factor map. 
\end{prop}

\begin{proof}
Let $T:X \rightarrow Y$ be the factor map. Notice that $T$ induces a map on the Bochner spaces of essentially bounded functions, namely
$$
T^\ast:\Linf(Y,E) \rightarrow \Linf(X,E) \ , \ \ T^\ast(\lambda):=\lambda \circ T \ ,
$$
and more generally for any Bochner $\mathrm{L}^p$-space. 

By applying twice Theorem \ref{theorem_expo}, we can deduce that the map
\begin{equation} \label{eq_change_coefficients_factor}
\Hb^2(\Gamma,\Linf(Y,E)) \rightarrow \Hb^2(\Gamma,\Linf(X,E)),
\end{equation}
induced by $T$ at the level of coefficients is conjugated to $\Hmb^2(f_T)$. As a consequence it is sufficient to show that the map of Equation \eqref{eq_change_coefficients_factor} is injective. Notice that we dropped the weak$^\ast$ measurability requirement on the essentially bounded functions in virtue of \cite[Lemma 3.3.3]{monod:libro}.

We can apply to both sides of Equation \eqref{eq_change_coefficients_factor} the change of coefficients from essentially bounded to square-integrable functions, namely we have a commutative diagram:
$$
\xymatrix{
\Hb^2(\Gamma,\Linf(Y,E)) \ar[rr] \ar[d]^{\kappa_Y} && \Hb^2(\Gamma,\Linf(X,E)) \ar[d]^{\kappa_X} \\
\Hb^2(\Gamma,\mathrm{L}^2(Y,E)) \ar[rr] && \Hb^2(\Gamma,\mathrm{L}^2(X,E))\ ,
}
$$
where $\kappa_Y$ (respectively $\kappa_X$) is the map induced by the change of coefficients from $\mathrm{L}^\infty$ to $\mathrm{L}^2$ on $Y$ (respectively on $X$). Both $\kappa_Y$ and $\kappa_X$ are injective by \cite[Corollary 9]{burger2:articolo}. In fact, the inclusion 
$$
\Linf(Y,E) \rightarrow \mathrm{L}^2(Y,E) \ ,
$$
is adjoint and $\mathrm{L}^2(Y,E)$ is separable by hypothesis (and the same holds for $X$). In a similar way, the map 
$$
T^\ast:\mathrm{L}^2(Y,E) \rightarrow \mathrm{L}^2(X,E) \ , 
$$
is injective and adjoint. Thus the bottom arrow is injective, still by \cite[Corollary 9]{burger2:articolo}. By the commutativity of the diagram the statement follows. 
\end{proof}

We now move to the case of homogeneous semidirect products for locally compact groups. In this case we will see that we have a variant of the Eckmann-Shapiro induction isomorphism relating the measurable cohomology of the semidirect product with the usual continuous bounded cohomology of the stabilizer of a point. 

\begin{rec_prop}[\ref{proposition_homogeneous_semidirect_product}]
Let $G$ be a locally compact group and let $H<G$ be a closed subgroup. Let $E$ be a coefficient $G$-module in the sense of Monod. If we endow $E$ with the structure of measurable coefficient $G \ltimes G/H$-module coming from the $G$-action, then we have isometric isomorphisms
$$
\Hmb^\bullet(G \ltimes G/H,E) \cong \Hcb^{\bullet}(H,E) \ .
$$
In particular, when $H$ is amenable the cohomology vanishes identically when the degree is greater than or equal to one. 
\end{rec_prop}

\begin{proof}
By Theorem \ref{theorem_expo} we know that 
$$
\Hmb^\bullet(G \ltimes G/H,E) \cong \Hcb^\bullet(G,\Linfw(G/H,E)) \ ,
$$
where the $G$ action on the module $\Linfw(G/H,E)$ comes from the structure of coefficient $G$-module of $E$. 

Since $E$ is a coefficient $G$-module, as observed in \cite[Remark 10.1.2 (v)]{monod:libro} we have an isometric isomorphism
$$
\mathrm{I}^G_HE:=\Linfw(G,E)^H \cong \Linfw(G/H,E) \ ,
$$
where $\mathrm{I}^G_HE$ is the induction module defined by Monod \cite[Definition 10.1.1]{monod:libro} endowed with the right translation $G$-action
$$g\cdot \lambda (h)\coloneqq \lambda(hg)\,.$$
In virtue of the Eckmann-Shapiro induction \cite[Proposition 10.1.3]{monod:libro} we have isometric isomorphisms
$$
\Hcb^\bullet(G,\Linfw(G/H,E)) \cong \Hcb^\bullet(G,\mathrm{I}^G_HE)  \cong \Hcb^\bullet(H,E) \ ,
$$
and the first part of the theorem is proved. When $H$ is amenable, the second part follows from \cite[Corollary 7.5.11]{monod:libro}.
\end{proof}

\begin{cor}\label{cor_stuck_zimmer}
Let $G$ be a connected simple Lie group of rank at least two and $E$ be a coefficient $G$-module in the sense of Monod. Let $(X,\mu)$ be a Lebesgue $G$-space and assume that the action is not essentially free. Then, there exists a lattice $\Gamma < G$ such that 
$$
\Hmb^\bullet(G \ltimes X,E) \cong \Hb^\bullet(\Gamma,E) \ . 
$$
\end{cor}

\begin{proof}
We know that the action of $G$ on $X$ is either essentially transitive or properly ergodic. The second possibility is ruled out by the Stuck-Zimmer theorem \cite[Theorem 2.1]{stuck:zimmer}, otherwise the action would be essentially free. As a consequence we know that $G \curvearrowright X$ must be essentially transitive. This means that, up to passing to an inessential contraction, our semidirect groupoid has the form $G \ltimes G/\Gamma$, for some $\Gamma < G$ closed subgroup. The existence of a finite $G$-invariant measure on $G/\Gamma$, together with the Borel density theorem \cite[Theorem 3.2.5]{zimmer:libro}, implies that $\Gamma$ must be a lattice. 

By applying Proposition \ref{proposition_homogeneous_semidirect_product} we obtain 
$$
\Hmb^\bullet(G \ltimes X, E) \cong \Hmb^\bullet(G \ltimes G/\Gamma,E) \cong \Hb^\bullet(\Gamma,E) \ ,
$$
and this concludes the proof. 
\end{proof}

The amenability statement of the Proposition \ref{proposition_homogeneous_semidirect_product} can be generalized to the more general setting of amenable spaces. In fact we have the following 

\begin{prop}\label{proposition_amenable_space}
Let $G$ be a locally compact group and let $E$ be a coefficient $G$-module in the sense of Monod. Let $S$ be a standard Borel amenable $G$-space. Then 
$$
\Hmb^\bullet(G \ltimes S,E) \cong 0
$$ 
when the degree is greater than or equal to one. 
\end{prop}

\begin{proof}
In virtue of our Theorem \ref{theorem_expo} we have that 
$$
\Hmb^\bullet(G \ltimes S,E) \cong \Hcb^\bullet(G,\Linfw(S,E)) \ . 
$$
By \cite[Theorem 5.7.1]{monod:libro} the amenability of $S$ implies that the coefficient $G$-module $\Linfw(S,E)$ is relatively injective. By \cite[Proposition 7.4.1]{monod:libro} we must have that the cohomology 
$$
\Hcb^\bullet(G,\Linfw(S,E)) \cong 0
$$
when the degree is greater than or equal to one.  
\end{proof}

\begin{es}
Since the action of a locally compact group $G$ on its Furstenberg boundary $B$ is amenable, for every coefficient $G$-module $E$ we have that 
$$
\Hmb^\bullet(G \ltimes B,E) \cong 0
$$
whenever the degree is greater than or equal to $1$, thanks to Proposition \ref{proposition_amenable_space}.
\end{es}

%\todo{A: Qui uso due volte la definizione di coefficient G module nel senso di Monod. Dovremmo definirlo qui o basta rimandare alla referenza? Idem per l'amenabilita' di uno spazio: visto che la facciamo prima dei gruppoidi amenabili}
%
%\todo{\color{red} F: Per me basta la referenza a Monod. Ne parliamo comunque nel Example 3.4 e nel Remark 4.2. Idem per azioni amenabili, anche perche' e' una nozione stranota. Magari metterei: "in the sense of Zimmer..."  }

Another important class of amenable groupoids is given by amenable $\mathrm{II}_1$-relations (see Example \ref{remark_feldman_moore_theorem}).

%\todo{F: Il risultato di Ornstein e Weiss e' generalizzato da Connes e Feldmann a tutte le relazioni di equivalenza amenabili , non necessariamente ergodiche. Loro provano che sono hyperfinite, cioe sono unione di azioni di Z. A me non e' chiaro come usare questa cosa in coomologia. Potremmo pensare a come si comporta la coomologia rispetto a iperfinitezza... Anche se poi e' tutto contenuto nell'amenabilita', non so se ne vale la pena.}

\begin{prop}\label{proposition_ergodic_amenable_relation}
Let $\calR$ be an amenable $\mathrm{II}_1$-relation. For every coefficient $\calR$-module $E$ endowed with the trivial action, we have that 
$$
\Hmb^\bullet(\calR,E) \cong 0
$$
when the degree is greater or equal than one. 
\end{prop}

\begin{proof}
Any amenable $\mathrm{II}_1$-relation is equivalent to the orbital relation of the action of $\mathbb{Z}$ on some ergodic space $(X,\mu)$ by Ornstein and Weiss \cite{OW80}. We claim that the action must be essentially free. By contradiction, suppose that there exists a positive measure set $U$ on which the field of stabilizers is not trivial. This would imply that the associated equivalence classes in $X$ would be finite, contradicting the fact that $\calR$ is countable. 

Since the cohomology is invariant in the (weak) isomorphism class, we have that 
$$
\Hmb^\bullet(\calR,E) \cong \Hmb^\bullet(\mathbb{Z} \ltimes X,E) \ ,
$$
and by Theorem \ref{theorem_expo} the right-hand side is given by
$$
\Hmb^\bullet(\matZ \ltimes X,E) \cong \Hb^\bullet(\mathbb{Z},\Linfw(X,E)) \ .
$$
The statement follows by the amenability of $\mathbb{Z}$ and by \cite[Corollary 7.5.11]{monod:libro}. 
\end{proof}

We want to underline that both the case of amenable $G$-spaces in Proposition \ref{proposition_amenable_space} and the case of amenable $\mathrm{II}_1$-relations in Proposition \ref{proposition_ergodic_amenable_relation} are particular instances of a more general fact that we are going to prove in Theorem \ref{theorem_amenability}.
In fact, we will see that for amenable measured groupoids (Definition \ref{definition_amenable_groupoid}) the measurable cohomology vanishes.

As already mentioned in the introduction, explicit computations of bounded cohomology for groups are difficult to obtain. Notice that bounded acyclicity is just a necessary condition for amenability, and the two notions becomes equivalent in the discrete case only if one requires the vanishing of the bounded cohomology with \emph{all} dual Banach modules \cite[Theorem 2.5]{Johnson}.
A recent striking result proved by Monod shows that the bounded cohomology of the \emph{Thompson group} $F$ with coefficients in any separable dual Banach $F$-module vanishes in all positive degrees \cite[Theorem 2]{monod:22}. 
This may be interpreted as a possible clue of the amenability of $F$, which remains an outstanding open question.
We want now to exploit both Monod's result and Theorem \ref{theorem_expo} to show that any semidirect groupoid associated to the Thompson group has trivial $2$-cohomology.

Let $\Gamma$ be a discrete group such that its bounded cohomology vanishes in degree $2$ for any separable dual Banach space. We call this property $(\mathbf{Ba2})$. 

\begin{cor}\label{corollary_thompson}
Let $\Gamma$ be a discrete group satisfying $(\mathbf{Ba2})$, let $S$ be a Lebesgue $\Gamma$-space and consider a measurable separable dual Banach $\Gamma \ltimes S$-space $E$. Then 
$$\Hmb^{2}(\Gamma \ltimes S,E)\cong 0 \ .$$
In particular, the statement holds for the Thompson group $F$. 
\end{cor}

\begin{proof}
By Theorem \ref{theorem_expo} it holds
$$
\Hmb^2(\Gamma \ltimes S, E) \cong \Hb^2( \Gamma ,\Linf(S,E)) \ ,
$$ 
where we dropped the weak$^*$ measurability on coefficient module of the right-hand side by \cite[Lemma 3.3.3]{monod:libro}. 

The coefficient module $\Linf(S,E)$ is only semi-separable, so we cannot apply directly condition $(\mathbf{Ba2})$. In any case, we know that the inclusion 
$$
\Linf(S,E) \rightarrow \mathrm{L}^2(S,E) \ 
$$
is adjoint and it induces an injective map 
$$
\Hb^2(\Gamma,\mathrm{L}^\infty(S,E)) \rightarrow \Hb^2(\Gamma,\mathrm{L}^2(S,E)) \ ,
$$
by \cite[Corollary 9]{burger2:articolo}. Being $\mathrm{L}^2(S,E)$ a separable dual Banach $\Gamma$-module by the separability of $S$ and $E$, we can apply condition $(\mathbf{Ba2})$, so
$$
\Hb^2(\Gamma,\mathrm{L}^2(S,E)) \cong 0 \ .
$$ 
Since the Thompson group $F$ satisfies $(\mathbf{Ba2})$ by \cite[Theorem 2]{monod:22}, the statement follows. 
\end{proof}

\begin{oss}
More generally, given a discrete group $\Gamma$, suppose that $\Hb^k(\Gamma,H)=0$ for every semi-separable Banach $\Gamma$-module $H$. Then, for every Lebesgue $\Gamma$-space $S$, we have that
$$
\Hmb^k(\Gamma \ltimes S,E)=0,
$$
where $E$ is any measurable Banach $\Gamma \ltimes S$-module. This holds for instance in the case of the Thompson group $F$, as noticed by Monod \cite{monod:22}. 
\end{oss}

We conclude this section with some applications of Theorem \ref{theorem_expo} to the case of higher rank lattices. 

\begin{prop}\label{proposition_irreducible_higher_rank}
Let $\Gamma <G=\mathbf{G}(\mathbb{R})^\circ$ be a lattice in the real points of a connected, simply connected, almost simple $\mathbb{R}$-algebraic group of rank at least two. Let $(X,\mu)$ be an ergodic Lebesgue $\Gamma$-space.  If $\Hb^2(\Gamma,\mathbb{R}) \cong 0$ then
$$
\Hmb^2(\Gamma \ltimes X,\mathbb{R}) \cong 0 \ ,
$$
where $\mathbb{R}$ is endowed with the trivial module structure. 
\end{prop}

\begin{proof}
In virtue of Theorem \ref{theorem_expo} it is sufficient to show that 
$$
\Hb^2(\Gamma,\mathrm{L}^\infty(X,\mathbb{R})) \cong 0 \ .
$$

Since the Banach $\Gamma$-module $\mathrm{L}^\infty(X,\mathbb{R})$ is semi-separable, we can apply \cite[Corollary 1.6]{Mon10} to get the isomorphism
$$
\Hb^2(\Gamma,\mathrm{L}^\infty(X,\mathbb{R})) \cong \Hb^2(\Gamma,\mathrm{L}^\infty(X,\mathbb{R})^\Gamma) \ .
$$
By the ergodicity of the $\Gamma$-action on $(X,\mu)$ and the vanishing assumption on the degree two bounded cohomology of $\Gamma$ we get
$$
\Hb^2(\Gamma,\mathrm{L}^\infty(X,\mathbb{R})^\Gamma) \cong \Hb^2(\Gamma,\mathbb{R}) \cong 0, 
$$
and the statement follows. 
\end{proof}

\begin{es}\label{es_higher_rank}
Let $\Gamma <G=\mathbf{G}(\mathbb{R})^\circ$ be a lattice in the real points of a connected, simply connected, almost simple $\mathbb{R}$-algebraic group of rank at least two. Given an ergodic Lebesgue $\Gamma$-space $(X,\mu)$, Proposition \ref{proposition_irreducible_higher_rank} suggests that in order to understand the vanishing of $\Hmb^2(\Gamma \ltimes X,\mathbb{R})$ it is crucial to compute the ordinary bounded $2$-cohomology of $\Gamma$. A striking result in that direction was given by Burger and Monod \cite[Theorem 21]{burger2:articolo}: in this context, they proved that the comparison map $\Hb^2(\Gamma,\mathbb{R}) \rightarrow \mathrm{H}^2(\Gamma,\mathbb{R})$ is actually injective. As a consequence it is sufficient to require that the ordinary $2$-cohomology of $\Gamma$ vanishes. This is the case when $\Gamma < \mathrm{Isom}(\mathcal{Y})$ is a torsion-free cocompact lattice, where $\mathcal{Y}$ is a symmetric space not of Hermitian type and the rank of $\mathrm{Isom}(\mathcal{Y})$ is at least $3$ \cite[Corollary 1.6]{BM1}. 
\end{es}

We now move to the case of lattices in products. More precisely we will consider a lattice 
$$
\Gamma < G=\prod_{i=1}^k G_i \ ,
$$
where $k \geq 2$ and each $G_i$ is locally compact second countable group satisfying $\Hcb^2(G_i;\mathbb{R})\cong 0$. We set
$$
G'_i:=\prod_{j \neq i} G_j 
$$ 
for $i=1,\ldots,k$. We say that a standard Borel $G$-space $(X,\mu)$ is \emph{irreducible} if each $G'_i$ acts ergodically on $X$. Following Burger and Monod \cite{burger2:articolo} we say that $\Gamma$ is \emph{irreducible} if each projection of $\Gamma$ on $G_i$ is dense. 

\begin{prop}\label{proposition_lattice_products}
Let $k \geq 2$. Consider an irreducible lattice $\Gamma < \prod_{i=1}^k G_i$ in a product of locally compact second countable groups satisfying $\Hcb^2(G_i,\mathbb{R}) \cong 0$, for $i=1,\ldots,k$. Let $(X,\mu)$ be an irreducible standard Borel $G$-space. Then it holds that 
$$
\Hmb^2(\Gamma \ltimes X, \mathbb{R}) \cong 0 \ ,
$$
where $\mathbb{R}$ is endowed with the trivial module structure. 
\end{prop}

\begin{proof}
Again, by Theorem \ref{theorem_expo}, it is sufficient to show that 
$$
\Hb^2(\Gamma,\mathrm{L}^\infty(X,\mathbb{R})) \cong 0 \ .
$$
Notice that the inclusion 
$$
\mathrm{L}^\infty(X,\mathbb{R}) \rightarrow \mathrm{L}^2(X,\mathbb{R}) \ 
$$
induces an injective map 
$$
\Hb^2(\Gamma,\mathrm{L}^\infty(X,\mathbb{R})) \rightarrow \Hb^2(\Gamma,\mathrm{L}^2(X,\mathbb{R})) \ ,
$$
by \cite[Corollary 9]{burger2:articolo}. By \cite[Theorem 14]{burger2:articolo}, the following decomposition holds
$$
\Hb^2(\Gamma,\mathrm{L}^2(X,\mathbb{R})) \cong \bigoplus_{i=1}^k \Hcb^2(G_i,\mathrm{L}^2(X,\mathbb{R})^{G_i'}) \cong \bigoplus_{i=1}^k \Hcb^2(G_i,\mathbb{R}) \ ,
$$
where the isomorphism on the right-hand side holds because of the irreducibility of the space $(X,\mu)$. By the vanishing assumption on each $\Hcb^2(G_i;\mathbb{R})$ the statement follows. 
\end{proof}

\begin{es}\label{es_lattice_products}
The vanishing condition given in Proposition \ref{proposition_lattice_products} can be easily verified when $\Gamma < \prod_{i=1}^k G_i$ is an irreducible lattice in a product of simple Lie groups. In fact, in that case, it is sufficient to require that none of the factors is of Hermitian type. 
\end{es}

\begin{es}\label{es_non_vanishing_cohom}
So far we have seen several cases of semidirect groupoids for which the measurable bounded $2$-cohomology vanishes. Here we want to point out that it is not complicated to build explicit examples of semidirect groupoids with non-trivial $2$-cohomology. In fact, let  $\Gamma \leq \mathrm{PU}(n,1)$ be a torsion-free complex hyperbolic lattice and let $(X,\mu)$ be a Lebesgue $\Gamma$-space. By Theorem \ref{theorem_expo} we know that 
$$
\Hmb^2(\Gamma \ltimes X,\mathbb{R}) \cong \Hb^2(\Gamma,\Linf(X,\mathbb{R})) \ .
$$
By a slight modification of \cite[Corollary 2.6]{MonShal0} we have the following isomorphism 
$$
\Hb^2(\Gamma,\Linf(X,\mathbb{R})) \cong \mathcal{Z}\mathrm{L}_{\mathrm{w}^\ast,\mathrm{alt}}((\partial_\infty \mathbb{H}^n_{\mathbb{C}})^3,\Linf(X,\mathbb{R}))^\Gamma \ ,
$$
where the right-hand side denotes the space of alternating cocycles. A proper subspace of the latter, corresponding to the one of $\mathrm{PU}(n,1)$-invariant alternating cocycles, is generated by the Cartan function (see for instance \cite{BICartan}). As a consequence the measurable bounded $2$-cohomology of $\Gamma \ltimes X$ is not trivial. 

Notice that Theorem \ref{theorem_expo_orbit_equivalence} ensures that the vanishing of the measurable bounded cohomology of an ergodic semidirect groupoid $\Gamma \ltimes X$ is a weakly orbit equivalence invariant. As a consequence, the ergodic action $\Gamma \curvearrowright X$ cannot be weakly orbit equivalent to any other ergodic action $\Lambda \curvearrowright Y$ with  $\Hmb^2(\Lambda \ltimes Y,\mathbb{R})\cong 0$ (for instance any lattice of Example \ref{es_higher_rank} and Example \ref{es_lattice_products}). 
\end{es}

\begin{oss}\label{remark_vanishing_parametrized}
The importance of both Proposition \ref{proposition_irreducible_higher_rank} and Proposition \ref{proposition_lattice_products} becomes more transparent when we relate them to the existence of some families of Borel 1-cocycles. More precisely, the authors \cite{sarti:savini:2} have recently introduced the notion of parametrized K\"{a}hler class for Borel $1$-cocycles with values into a Hermitian Lie group. They proved that such a class determines the cohomology class of a Zariski dense $1$-cocycle and the vanishing results contained in both Proposition \ref{proposition_irreducible_higher_rank} and Proposition \ref{proposition_lattice_products} imply that there are no Zariski dense cocycles for higher rank lattices. Something similar has been proved by the second author \cite{savini:euler} in the context of non-elementary cocycles with values in the orientation preserving homeomorphisms of the circle.   
\end{oss}

\section{Amenability}\label{section_Amenability}

This section is devoted to prove Theorem \ref{theorem_amenability}, that is the vanishing of the measurable bounded cohomology of amenable groupoids.
To this end, we need to introduce the notion of amenability, that is given in terms of the existence of an invariant system of means.
Other equivalent definitions have been introduced in the book by Anantharaman-Delaroche and Renault \cite{delaroche:renault}, and our choice is simply motivated by the arguments used to prove Theorem \ref{theorem_amenability}. In fact, the proof follows the same line adopted for instance in Frigerio's book \cite{miolibro}, and the idea is to construct a contracting homotopy for the cochain complex that computes the bounded cohomology of the groupoid. In our framework, thanks to the disintegration isomorphism of Equation \eqref{equation_disintegration_isomorphism}, this can be done fiberwisely, namely on each $\Linfw((\calG^{(\bullet)},\nu_x^{(\bullet)}),E)$ for every $x\in \calG^{(0)}$.

We start with the following
\begin{defn}\label{definition_invariant_system_of_means}
Let $\calG$ be a measured groupoid and let $\mu$ be the probability measure on the unit space $\calG^{(0)}$. Denote by $\nu$ a symmetric and quasi-invariant probability measure on $\calG$ and by $\{\nu^x\}_{x\in \calG^{(0)}}$ its disintegration along the target $t:\calG\rightarrow \calG^{(0)}$. An \emph{invariant measurable system of means} is a family $\{\frakm^x\,|\, x\in \calG^{(0)}\}$ of linear functions $\frakm^x: \Linf(\calG,\nu^x)\rightarrow \matR$ such that 
the map $x\mapsto \frakm^x(\lambda)$ is Borel measurable for every $\lambda\in \Linf(\calG)$
and for $\nu$-almost every $g\in \calG$ one has
\begin{equation}\label{equation_invariance_mean}
\frakm^{t(g)} (\lambda^{t(g)})=\frakm^{s(g)}(g^{-1}\lambda^{t(g)})\,.
\end{equation}
where $\lambda^x$ is the element in $\Linf(\calG,\nu^x)$ defined by the isomorphism of Equation \eqref{equation_disintegration_isomorphism} and $g^{-1}\lambda^{t(g)}$ is defined analogously as in Equation \eqref{equation_action_section}.
\end{defn}

\begin{defn}\label{definition_amenable_groupoid}
A measured groupoid $\calG$ is \emph{amenable} if it admits an invariant measurable system of means.
\end{defn}

\begin{es}
When $\calG=G \ltimes X$ is the semidirect groupoid of a measure preserving action of a locally compact group $G$, the amenability of $\calG$ is equivalent to the amenability of the action in the sense of Zimmer \cite{zimmer:libro}, as observed also by Anantharaman-Delaroche and Renault \cite[Examples 3.2.2]{delaroche:renault}.
\end{es}

In \cite[Chapter 3.1.c]{delaroche:renault} the authors introduce a more general notion of system of means for a Borel surjection between $\calG$-spaces, where the latter are measure spaces endowed with a measurable $\calG$-action such that the semidirect groupoids are measured, as recalled in Example \ref{example_groupoid_action_measurable}. More precisely, let $(Y,\nu)$ and $(X,\mu)$ be two standard Borel probability spaces with a measure class preserving $\calG$-action, for a measured groupoid $\calG$. Suppose we have a $\calG$-equivariant Borel surjection $\pi:Y \rightarrow X$ such that $\pi_\ast\nu = \mu$ and let $\{ \nu^x\}_{x \in X}$ be the disintegration of $\nu$ along $\pi$. Notice that the disintegration is \emph{quasi} $\calG$-\emph{invariant}, that means $g\nu^{x} \sim \nu^{gx}$ \cite[pg. 53-54]{delaroche:renault}. 

An \emph{invariant system of means} for the surjection $\pi:Y \rightarrow X$ \cite[Definition 3.1.26]{delaroche:renault} is a family $\{ \frakm^x \}_{x \in X}$ of linear operators $\frakm^x:\Linf(Y,\nu^x) \rightarrow \mathbb{R}$ such that, for every $\lambda \in \Linf(Y,\nu)$, we have that $x \mapsto \frakm^x(\lambda)$ is Borel measurable and 
\begin{equation}\label{eq ism projection}
\frakm^{g^{-1}x}(g^{-1}\lambda)=\frakm^x(\lambda) \ ,
\end{equation}
for almost every $(g,x) \in  \calG \ltimes X$. Here we are tacitly considering the structure of measured groupoid on $\calG \ltimes X$ that we mentioned in Example \ref{example_groupoid_action_measurable}. We want to stress that, in this paper, all these definitions will be applied only in the proof of Theorem \ref{theorem_amenability}, when $\pi$ is the target map $t:\calG\rightarrow \calG^{(0)}$ and
the measure structure will be the one of $\calG$.
In this case, Equation \eqref{eq ism projection} boils down to Equation \eqref{equation_invariance_mean} and the sentence ``almost every $(g,x)$" refers, up to renormalization, to the convolution of the Borel Haar system $(\rho^x)_{x\in \calG^{(0)}}$ with $\mu$, which coincides with a symmetric invariant measure on $\calG$.
This means that Definition
 \ref{definition_invariant_system_of_means} is a particular case of
 \cite[Definition 3.1.26]{delaroche:renault}.

Given an invariant system of means for a Borel surjection $\pi:Y \rightarrow X$, one can build a bounded linear operator $$\frakm \in B_{\Linf(\calG^{(0)})}(\Linf(Y),\Linf(X))\,,$$ called \emph{mean}, as follows 
\begin{equation}\label{equation_mean_system}
 \frakm(\lambda) (x) \coloneqq \frakm^x(\lambda)\,.
 \end{equation} 
Moreover, the invariance of the system $\{\frakm^x \}_{x \in X}$ implies for $\frakm$ the following stability property
$$
\frakm(\Linf(Y)^\calG) \subset \Linf(X)^\calG \,.
$$
Here $\Linf(Y)^\calG$ and $\Linf(X)^\calG$ denote the spaces of $\calG$-invariant functions on $Y$ and $X$, respectively, and they are defined as in Section \ref{section_Measurable bounded cohomology}. Such property, called \textbf{(INV0)} by Anantharaman-Delaroche and Renault \cite[Definition 3.1.13]{delaroche:renault}, can be actually strengthened to any fiber product involving $Y$ and $X$. Precisely, given another $\calG$-equivariant Borel surjection $q:S \rightarrow X$ of a standard Borel probability $\calG$-space $(S,\tau)$ such that $q_*\tau = \mu$, we can define the fiber product $Y \ast S$ with respect to $\pi$ and the fiber product $X \ast S$ with respect to the identity on $X$. The $\calG$-action can be extended to the fiber products and we say that the mean $\frakm$ satisfies \textbf{(INV1)} if, for every such $(S,\tau)$, one has
$$
(\frakm \otimes \mathbbm{1}_S) (\Linf(Y \ast S)^\calG) \subset \Linf(X \ast S)^\calG ,
$$
where $\frakm \otimes \mathbbm{1}_S$ is the \emph{extension} of the mean $\frakm$ to the fiber product $Y \ast S$ \cite[Chapter 1.3.b]{delaroche:renault}. The latter can be constructed by observing that the projection on the second factor $p_2:Y \ast S \rightarrow S$ is a $\calG$-equivariant Borel surjection and we can disintegrate $\nu \ast \tau$ along $p_2$ obtaining the family of measures $\{ \nu^{q(s)} \otimes \delta^s \}_{s \in S}$ \cite[Section 2]{ramsay}, where $\delta^s$ denotes the Dirac measure centred in $s \in S$. This fact allows to identify $\Linf(Y \ast T, \nu^{q(s)} \otimes \delta^s)$ with $\Linf(Y,\nu^{q(s)})$, and to consider the invariant system of means $\{\frakm^{(x,s)}\}_{(x,s) \in X \ast S}$ which acts on $\Linf(Y \ast S, \nu^{q(s)} \otimes \delta^s)$ as $\frakm^{q(s)}$. Finally for any $\lambda \in \Linf(Y \ast S, \nu \ast \tau)$ it is sufficient to define
\begin{equation}\label{eq extension mean}
(\frakm \otimes \mathbbm{1}_S)(\lambda)(x,s):=\frakm^{(x,s)}(\lambda) \ .
\end{equation}

With this definition, Anantharaman-Delaroche and Renault proved the equivalence between the existence of an invariant system of means and the existence of a mean with property \textbf{(INV1)}.
This different point of view will be used together with the disintegration isomorphism of Equation \eqref{equation_disintegration_isomorphism} in the proof of Theorem \ref{theorem_amenability}, in order to pass from a \emph{global} approach to a \emph{fiberwise} computation.

%The fact that $\{m^x\}_{x \in \calG^{(0)}}$ is an invariant system implies that $m$ is a $\calG$-invariant mean, in the sense that 
%\begin{equation}\label{equation_invariance_mean2}
%m(f*\lambda)=f* m(\lambda) \ ,
%\end{equation}
%where $\lambda\in \Linf(\calG)$, $f:\calG\rightarrow \matR$ is a Borel function such that $\nu(|f|)$ is bounded and the convolution in Equation \eqref{equation_invariance_mean2} is defined on the right-hand side as 
%$$f*\lambda  (g)\coloneqq \int_{\calG^{t(g)}} f(h)\lambda(h^{-1} g)d\nu^{t(g)}(h) $$
%for $\lambda\in\Linf(\calG)$ and on the left-hand side as 
%$$f*\sigma  (x)\coloneqq \int_{\calG^{x}} f(h)\sigma(s(h))d\nu^{x}(h) $$
%for $\sigma\in \Linf(\calG^{(0)})$.

%Notice that if $f=\mathbbm{1}_x$ for some $x\in \calG^{(0)}$ where $$\mathbbm{1}_x(g)\coloneqq \begin{cases}  1 & \text{ if } t(g)=x\\ 0 & \text{ otherwise}\,.\end{cases}$$
%is the characteristic function of $\calG^x$, then Equation \eqref{equation_invariance_mean2} boils down to Equation \eqref{equation_invariance_mean}.
%Precisely we have 
%\begin{align*}
%m(\mathbbm{1}*\lambda)(x)=m^x((\mathbbm{1}_x*\lambda)^x)&=
%m^x\left(g\mapsto \int_{\calG^x} \lambda(h^{-1} g) d\nu^x(h)\right)\\
%&=m^x\left(g\mapsto \int_{\calG^x} h\cdot \lambda^{s(h)}(g) d\nu^x(h)\right)\\
%&= \int_{\calG^x} m^{t(h)} \left(h\cdot \lambda^{s(h)}\right)  d\nu^x(h)
%\end{align*}
%and 
%$$\mathbbm{1}*m(\lambda)(x)= \int_{\calG^x} m(\lambda)(s(h)) d\nu^x(h)=
% \int_{\calG^x} m^{s(h)}\left(\lambda^{s(h)}\right) d(h)\,.$$

\begin{rec_thm}[\ref{theorem_amenability}]
Let $\calG$ be an amenable measured groupoid and $E$ a measurable coefficient $\calG$-module. Then 
$$\Hmb^{\bullet}(\calG,E)\cong 0$$
whenever the degree is greater or equal than one. 
\end{rec_thm}

 \begin{proof}

The strategy of the proof is to construct a contracting homotopy $k$ for the complex defining the bounded cohomology of $\calG$. We will exploit the amenability of $\calG$  to define the homotopy. Let $\{\frakm^x\}_{x \in \calG^{(0)}}$ be an invariant system of means on $\calG$. Recall that $E$ is a measurable coefficient $\calG$-module, hence it coincides with the contragradient representation of $\calG$ on some separable Banach space $(E^\flat,L^\flat)$. 

Thanks to Equation \eqref{equation_mean_system}, we can define a global mean $\frakm:\Linf(\calG) \rightarrow \Linf(\calG^{(0)})$. By choosing as $(S,\tau)=(\calG^{(\bullet-1)},\nu^{\bullet-1})$, for $\bullet > 1$, we can define the extended mean 
$$
\frakm \otimes \mathbbm{1}_{\calG^{(\bullet-1)}}: \Linf(\calG^{(\bullet)}) \rightarrow \Linf(\calG^{(\bullet-1)}) 
$$
with the help of Equation \eqref{eq extension mean}. Recall that the extensions will preserve $\calG$-invariance. 

Let $\lambda \in \Linfw(\calG^{(\bullet)},E)$ be a weak$^*$ Borel function. We consider $v \in E^\flat$ and the function
$$\lambda_v(\ \cdot \ )\coloneqq   \langle \lambda(\ \cdot \ ),v\rangle   \in \Linf(\calG^{(\bullet)})\,,$$
where the pairing is the one on $E\otimes E^\flat $.
Then, for any $\sigma \in \mathrm{L}^1(\calG^{(\bullet-1)})$, we define the bilinear form
$$
k^{\bullet-1}\lambda(\sigma,v):=\langle (\frakm \otimes \mathbbm{1}_{\calG^{(\bullet-1)}})( \lambda_v ) | \sigma \rangle \ ,
$$
where the pairing is now
 the one on $\Linf(\calG^{(\bullet-1)}) \otimes \mathrm{L}^1(\calG^{(\bullet-1)})$. Notice that the definition of $k^{\bullet-1}\lambda$ is well-posed because we can find a countable dense subset of vectors $v\in E^{\flat}$, whose existence is ensured by the separability of $E^{\flat}$. The bilinear form $k^{\bullet-1}\lambda$ is bounded, linear with respect to $\lambda$ and it induces a linear functional on $\mathrm{L}^1(\calG^{(\bullet-1)},E^\flat)$. In other words, with a slight abuse of notation and exploiting the duality isomorphism of Equation \eqref{equation_duality}, we can write
$$
k^{\bullet-1} \lambda \in \Linfw(\calG^{(\bullet-1)},E) \ .
$$
Using the disintegration isomorphism of Equation \eqref{equation_disintegration_isomorphism} we can reformulate our definition in terms of the invariant system of means $\{\frakm^x \}_{x \in \calG^{(0)}}$ and of the section associated to $\lambda$. 
More precisely, 
for $v \in E^\flat$, 
we can disintegrate 
$\lambda_v   \in \Linf(\calG^{(\bullet)})$ along the projection $\calG^{(\bullet)} \rightarrow \calG^{(\bullet-1)}$, getting a family 
of functions $$\lambda_v^{g_1,\ldots,g_{\bullet-1}}\in \Linf(\calG^{(\bullet)},\nu^{t^{\bullet-1}(g_1,\ldots,g_{\bullet-1})} \otimes \delta^{(g_1,\ldots,g_{\bullet-1})}) \cong \Linf(\calG,\nu^{t^{\bullet-1}(g_1,\ldots,g_{\bullet-1})})$$ defined by
$$
\lambda_v^{g_1,\ldots,g_{\bullet-1}}(g)\coloneqq\begin{cases}
\lambda_v(g,g_1,\ldots,g_{\bullet-1}) & \text{if } t(g)=t^{\bullet-1}(g_1,\ldots,g_{\bullet-1})\\
0  & \text{otherwise}\,.
\end{cases}
$$
Hence, by Equation \eqref{eq extension mean}, we obtain
\begin{equation*}
(\frakm \otimes \mathbbm{1}_{\calG^{(\bullet-1)}})(\lambda_v )(g_1,\ldots,g_{\bullet-1}) = \frakm^{t^{\bullet-1}(g_1,\ldots,g_{\bullet-1})}(\lambda_v^{g_1,\ldots,g_{\bullet-1}})
\end{equation*}
for almost every $(g_1,\ldots,g_{\bullet-1})\in \calG^{(\bullet-1)}$. Therefore, we can rewrite $k^{\bullet-1}\lambda$ fiberwisely as
\begin{equation}\label{eq kappa section}
\langle (k^{\bullet-1} \lambda)^x(g_1,\ldots,g_{\bullet-1}) , v \rangle:=\frakm^x(g \mapsto \langle \lambda^x(g,g_1,\ldots,g_{\bullet-1}), v \rangle) \ ,
\end{equation}
where $x=t^{\bullet-1}(g_1,\ldots,g_{\bullet-1})$ and Equation \eqref{eq kappa section} holds for every $v \in E^\flat$ and for almost every $(g_1,\ldots,g_{\bullet-1}) \in \calG^{ (\bullet-1)} $.  Thanks to the above considerations, from now on we will work only in terms of sections obtained via the disintegration isomorphism, that is using Equation  \eqref{eq kappa section}.

We need to show that $k^{\bullet-1}$ is a contracting homotopy and that it preserves $\calG$-invariance. 
We start by showing that it is a chain contraction. For every $v \in E^\flat$, on one hand we have that 

\begin{align*}
\langle (\delta^{\bullet-2}(k^{\bullet-1}\lambda))^x(g_1,\ldots,g_{\bullet}) , v \rangle&=
\sum\limits_{j=1}^{\bullet} (-1)^{j-1} \langle (k^{\bullet-1} \lambda)^x(g_1,\ldots,\widehat{g_j},\ldots,g_{\bullet}),v \rangle \\
&= \sum\limits_{j=1}^{\bullet} (-1)^{j-1} \frakm^x(g\mapsto \langle \lambda^x(g,g_1,\ldots,\widehat{g_j},\ldots,g_{\bullet}),v \rangle )
\end{align*}
and the other hand we have
\begin{align*}
\langle (k^{\bullet}(\delta^{\bullet-1}\lambda))^x(g_1,\ldots,g_{\bullet}), v \rangle&=
\frakm^x(g\mapsto \langle (\delta^{\bullet-1}\lambda)^x(g,g_1,\ldots,g_{\bullet}),v \rangle )\\
&=\frakm^x(g\mapsto \langle \lambda^x(g_1,\ldots,g_{\bullet}),v \rangle)\\
&+ \sum\limits_{j=1}^{\bullet}(-1)^{j} \frakm^x(g\mapsto \langle \lambda^x(g,g_1,\ldots,\widehat{g_j},\ldots,g_{\bullet}),v \rangle)\\
&=\langle \lambda^x(g_1,\ldots,g_{\bullet}),v \rangle\\
&+ \sum\limits_{j=1}^{\bullet}(-1)^{j} \frakm^x(g\mapsto \langle \lambda^x(g,g_1,\ldots,\widehat{g_j},\ldots,g_{\bullet}),v \rangle)
\end{align*}
where we exploited both the linearity of $\frakm^x$ and the fact that $\frakm^x$ preserves the constant functions.
As a consequence of the previous computation we obtain that 
$$\delta^{\bullet-2}\circ k^{\bullet-1} + k^{\bullet}\circ \delta^{\bullet-1} = \id_{\Linfw(\calG^{(\bullet)},E)}\, $$
for every $\bullet\geq 1$, which means that $k^{\bullet-1}$ is a contracting homotopy. 

We need to show that $k^{\bullet-1}$ preserve the $\calG$-invariance. Take a Borel function $\lambda$ whose class in $\Linfw(\calG^{(\bullet)},E)$ is $\calG$-invariant. Taking $g_1,\ldots, g_{\bullet}\in \calG$ and $v\in E^{\flat}$
we have
\begin{align*}
 & \langle g \cdot (k^{\bullet-1}(\lambda))^{s(g)}(g_1,\ldots,g_{\bullet-1}) , v\rangle \\
 =& \; \langle L(g) (k^{\bullet-1}(\lambda))^{s(g)}(g^{-1}g_1,\ldots,g^{-1}g_{\bullet-1}), v\rangle\\
 =&\; \langle (k^{\bullet-1}(\lambda))^{s(g)}(g^{-1}g_1,\ldots,g^{-1}g_{\bullet-1}) ,L^{\flat}(g^{-1})v\rangle)\\
=& \;\frakm^{s(g)}(h\mapsto \langle\lambda^{s(g)}(h,g^{-1}g_1,\ldots,g^{-1}g_{\bullet-1})  , L^{\flat}(g^{-1})v\rangle ) \, ,
\end{align*}
where we moved from the first line to the second one using the action of $g$ on the $s(g)$-fiber and we concluded using the duality between $E$ and $E^\flat$. 
Notice also that 
\begin{align*}
\langle \lambda^{s(g)}(h,g^{-1}g_1,\ldots,g^{-1}g_{\bullet-1}),L^\flat(g^{-1})v \rangle&=\langle L(g)\lambda^{s(g)}(h,g^{-1}g_1,\ldots,g^{-1}g_{\bullet-1}),v \rangle \\
&=\langle (g\lambda^{s(g)}(gh,g_1,\ldots,g_{\bullet-1}),v \rangle \\
&=\langle \lambda^{t(g)}(gh,g_1,\ldots,g_{\bullet-1}),v \rangle \\
&=g^{-1} \cdot \left( \langle \lambda^{t(g)}(h,g_1,\ldots,g_{\bullet-1}),v \rangle \right)\ ,
\end{align*}
where we moved from the second line to the third one using the $\calG$-invariance of $\lambda$ and the equality $g^{-1}gh=h$
(condition (iii) of Definition \ref{def_groupoid}).
We finally concluded using backward the definition of $g^{-1}\langle \lambda^{t(g)}( \ \cdot \ ), v \rangle$. 
We can move on in our computation as follows
\begin{align*}
&\; \frakm^{s(g)}(h\mapsto \langle\lambda^{s(g)}(h,g^{-1}g_1,\ldots,g^{-1}g_{\bullet-1})  , L^{\flat}(g^{-1})v \rangle) \\
=&\; \frakm^{s(g)}(h\mapsto  g^{-1} \cdot ( \langle \lambda^{t(g)}(h,g_1,\ldots,g_{\bullet-1},v \rangle)))\\
= &\; \frakm^{t(g)}(h\mapsto \langle  \lambda^{t(g)}(h,g_1,\ldots,g_{\bullet-1}),v\rangle)\\
=&\; \langle (k^{\bullet-1}(\lambda))^{t(g)}(g_1,\ldots,g_{\bullet-1}),v\rangle
\end{align*}
where the first equality follows by Equation \eqref{equation_invariance_mean}.
Hence $k^{\bullet-1}$ restricts to $\calG$-invariants and it defines a contracting homotopy for the subcomplex 
$(\Linfw(\calG^{(\bullet)},E)^{\calG},\delta^{\bullet})$. The statement follows.
 \end{proof}

%\todo{A: Fare un remark conclusivo in cui si fa il paragone con il nostro risultato, quello di Blank, quello di DeLaroche e quello di Monod per la relativa inviettivita' dei coefficient Linf in spazi amenabili} 

\begin{oss}\label{remark_amenability}
The connection between bounded cohomology and amenability has been studied in different contexts, for instance by Monod \cite{monod:libro} in case of groups, and by Blank \cite{blank} and Anantharaman-Delaroche \cite{delaroche:91} for groupoids. We list below some results connected with our Theorem \ref{theorem_amenability}.

For group actions, a characterization of amenability in terms relative injectivity has been proved by Burger and Monod  \cite[Theorem 1]{burger2:articolo}. Precisely, an action of a locally compact group $G$ on a Lebesgue space $S$ (or equivalently the semidirect groupoid $G\ltimes S$) is amenable if and only if the module $\Linfw(S,E)$ is \emph{relatively injective} for every coefficient $G$-module $E$. The previous result, together with Theorem \ref{theorem_expo} and \cite[Proposition 7.4.1]{monod:libro}, leads us to the bounded acyclicity statement given in Proposition \ref{proposition_amenable_space}. 

%In case of a discretised groupoid, Definition \ref{definition_amenable_groupoid} boils down to the notion of amenability given by Blank \cite[Definition 4.1.3]{blank}. Moreover, in virtue of Remark \ref{remark_discrete}, it follows that Theorem \ref{theorem_amenability} corresponds to vanishing result shown in \cite[Corollary 4.2.4]{blank}, where the measurable bundle of coefficients reduces to the constant module $E$.

Coming back to the measurable setting, the first bounded cohomology group of a measured groupoid has already been introduced by Anantharaman-Delaroche \cite{delaroche:91}. Although the author uses the inhomogeneous coboundary operator instead of the homogeneous one, it is immediate to check the equivalence with our definition. Moreover, it is proved that, for an amenable groupoid, each 1-cocycle is in fact a coboundary. The converse is also true for amenable actions under an additional technical hypothesis. However, all these results are stated in the more general setting of coefficient bundles of Banach spaces. For this reason, we do not go into details and we refer to \cite{delaroche:renault}.
%Measurable bounded cohomology with coefficients into bundles of Banach spaces is defined and studied in forthcoming paper of the two authors. 

\end{oss}

\begin{oss}\label{remark_amenable_actions}
In Section \ref{section_Main consequences of the exponential law} we prove two vanishing results in case of semidirect groupoids coming from amenable actions. In particular Proposition \ref{proposition_amenable_space} and Proposition \ref{proposition_ergodic_amenable_relation} state respectively that amenable actions of a locally compact group on a Lebesgue space and amenable $\mathrm{II}_1$-relations have trivial bounded cohomology. Both cases are particular instances of Theorem \ref{theorem_amenability}, since amenable actions produces amenable groupoids. We also point out that Proposition \ref{proposition_amenable_space} considers only coefficient $G$-modules in the sense of Monod, whereas Theorem \ref{theorem_amenability} holds for any measurable coefficient $(G\ltimes S)$-module (see also Remark \ref{remark_monod_shalom} for a similar discussion in the case of orbit equivalences).
 
\end{oss}

\section*{Declarations}
\subsection*{Data Availability Statement}
No new data were created or analysed in this study. Data sharing is not applicable to this article.

\subsection*{Competing interests} 
On behalf of all authors, the corresponding author states that there is no conflict of
interest.

\subsection*{Fundings}
 The first author's research is funded by MUR through the PRIN project ``Geometry and topology of manifolds" and
 supported by the GNSAGA--INdAM.
\bibliographystyle{alpha}

\bibliography{biblionote}

\end{document}